\documentclass[12pt]{article}
\usepackage[margin=0.8in]{geometry}
\usepackage{algcompatible}
\usepackage{caption}
\providecommand{\keywords}[1]{\small \quad \quad \textbf{Keywords:} #1}
\usepackage{titlesec}
\usepackage{hyperref}
\hypersetup{colorlinks=true, linkcolor=blue, filecolor=magenta, urlcolor=cyan, citecolor=blue}
\usepackage[affil-it]{authblk}
\usepackage{graphicx}
\usepackage{multirow}
\usepackage{amsmath,amssymb,amsfonts}
\usepackage{amsthm}
\usepackage{mathrsfs}
\usepackage[title]{appendix}
\usepackage{xcolor}
\usepackage{textcomp}
\usepackage{manyfoot}
\usepackage{booktabs}
\usepackage{longtable}
\usepackage{nicematrix}
\usepackage[numbers,sort&compress]{natbib}
\usepackage[linesnumbered,ruled,vlined]{algorithm2e}
\usepackage{algorithmicx}
\usepackage{algpseudocode}
\usepackage{listings}
\usepackage{xspace}
\usepackage{dsfont}
\usepackage{cleveref}
\usepackage{amsmath}
\usepackage{tikz}
\tikzstyle{vertex}=[draw,circle,minimum size=18pt,inner sep=0pt]
\usepackage[font=small, labelfont=bf, margin=1cm]{caption}
\usepackage{subcaption}
\usepackage{tabularx}
\usepackage{url}
\usepackage{makecell}
\usepackage{comment}
\usepackage{verbatim}
\usepackage{rotating}
\usepackage{xparse}
\usepackage{microtype}

\renewcommand{\citep}[1]{\hspace{-10em}\citep{#1}}
\RenewDocumentCommand\citep{ s O{} O{} m }{%
	\unskip\nobreak\hspace{0.15em}% 
	\IfBooleanTF{#1}
	{(\citet*[#2][#3]{#4})}%
	{(\citet[#2][#3]{#4})}%
}
\bibpunct{(}{)}{;}{a}{,}{,}                         

\DeclareMathOperator{\conv}{conv}

\newcommand{\R}{\mathbb{R}}
\newcommand{\X}{\mathcal{X}}
\newcommand{\Z}{\mathbb{Z}}

\newcommand{\D}{\delta}
\newcommand{\EPI}{EPI\xspace}
\newcommand{\EPIs}{EPIs\xspace}
\newcommand{\LEPI}{LEPI\xspace}
\newcommand{\LEPIs}{LEPIs\xspace}

\newcommand{\EF}{\texttt{KE}\xspace}
\newcommand{\EPIF}{\texttt{B\&C+EPI}\xspace}
\newcommand{\LEPIF}{\texttt{B\&C+LEPI}\xspace}
\newcommand{\GUB}{GUB\xspace}

\newcommand{\LP}{{LP}\xspace}
\newcommand{\MILP}{{MILP}\xspace}
\newcommand{\MINLP}{{MINLP}\xspace}

\newcommand{\SOCP}{{SOCP}\xspace}
\newcommand{\cGCMCLP}{MPCLP\xspace}
\newcommand{\MCLP}{MCLP\xspace}

\newcommand{\MPKPG}{MPKP-G\xspace}
\newcommand{\ISP}{ISP\xspace}
\newcommand{\CPX}{\texttt{CPX}\xspace}
\newcommand{\CPXEPI}{\texttt{CPX+EPI}'\xspace}
\newcommand{\CPXLEPI}{\texttt{CPX+LEPI}'\xspace}

\newcommand{\CO}{\mathcal{O}}
\newcommand\polymake{\texttt{polymake}\xspace}
\theoremstyle{plain}
\newtheorem{Theorem}{Theorem}[section]
\newtheorem{Corollary}[Theorem]{Corollary}
\newtheorem{Lemma}[Theorem]{Lemma}
\newtheorem{Remark}[Theorem]{Remark}
\newtheorem{Example}[Theorem]{Example}
\newtheorem{Definition}[Theorem]{Definition}

\newtheorem{Proposition}[Theorem]{Proposition}
\crefname{Theorem}{Theorem}{Theorems}
\crefname{Example}{Example}{Examples}
\crefname{Observation}{Observation}{Observations}
\crefname{Remark}{Remark}{Remarks}
\crefname{Proposition}{Proposition}{Propositions}
\crefname{Lemma}{Lemma}{Lemmas}
\crefname{Corollary}{Corollary}{Corollaries}
\crefname{subsection}{Section}{Sections}
\crefname{algorithm}{Algorithm}{Algorithms}
\crefname{figure}{Figure}{Figures}
\crefname{table}{Table}{Tables}
\crefname{section}{Section}{Sections}

\title{Polyhedral results for two classes of submodular sets with GUB constraints}

\author[a]{Weikang Qian}
\author[b]{Keyan Li}
\author[b]{Wei-Kun Chen}
\author[c,d]{Yu-Hong Dai}

\affil[a]{\small School of Mathematical Sciences, University of Science and Technology of China, Hefei 230026, China\\\textit{wk220@mail.ustc.edu.cn}}
\affil[b]{\small School of Mathematics and Statistics, Beijing Institute of Technology, Beijing 100081, China\\\textit{\{likeyan,chenweikun\}@bit.edu.cn}}
\affil[c]{\small Academy of Mathematics and Systems Science, Chinese Academy of Sciences, Beijing 100190, China}
\affil[d]{\small School of Mathematical Sciences, University of Chinese Academy of Sciences, Beijing 100049, China\\\textit{dyh@lsec.cc.ac.cn}}
\date{\small \today}

\begin{document}

\maketitle

\begin{abstract}
	In this paper, we investigate the polyhedral structure of two submodular sets with generalized upper bound (\GUB) constraints, which arise as important substructures in various real-world applications.
	We derive a class of strong valid inequalities for the two sets using sequential lifting techniques.
	The proposed lifted inequalities are facet-defining for the convex hulls of two sets and are stronger than the well-known extended polymatroid inequalities (\EPIs).
	We provide a more compact characterization of these inequalities and show that each of them can be computed in linear time.
	Moreover, the proposed lifted inequalities, together with bound and \GUB constraints, can completely characterize the convex hulls of the two sets, and can be separated using a combinatorial polynomial-time algorithm. 
	Finally, computational results on probabilistic covering location and multiple probabilistic knapsack problems demonstrate the superiority of the proposed lifted inequalities over the \EPIs within a branch-and-cut framework.
	\vspace{8pt} \\
\end{abstract}

\section{Introduction}\label{section1}
In this paper, we investigate the polyhedral structure of the epigraph of a concave function composed with a non-negative linear function under generalized upper bound (\GUB) constraints:
\begin{equation}\label{setX_0}
	X_0=\left\{ (w,x) \in \R \times\{0,1\}^{n}\,:\, 
	w \geq f(a^\top x),~ \sum_{i \in N_k} x_i \leq 1,~ \forall~ k\in K \right\},
\end{equation}
and its generalization with an additional linear function in the nonlinear constraint:
\begin{equation}\label{setX}
	X = \left\{ (w,x) \in \R \times\{0,1\}^{n}\,:\, 
	w \geq f(a^\top x) + b^\top x,~ \sum_{i \in N_k} x_i \leq 1,~ \forall~ k\in K \right\},
\end{equation}
where $f \,:\, \R \rightarrow \R$ is a concave function, 
$a\in \R^n_+$, $b\in \R^n$, 
and $\{N_k\}_{k\in K}$ is a partition of $[n]$; that is, $\bigcup_{k\in K} N_k = [n]$ and $N_{k_1}\cap N_{k_2} = \varnothing$ for any distinct $k_1, k_2\in K$ (throughout, for a non-negative integer $\tau$, we denote $[\tau]:=\{1,2,\dots,\tau\}$ with the convention that $[0]=\varnothing$).

The two sets $X_0$ and $X$ arise as important substructures in many real-world applications such as facility location \citep{Feldman1966, Hajiaghayi2003, Karatas2021}, probabilistic knapsack \citep{Zhang2018, Joung2020, Atamturk2013}, and approximate submodular minimization \citep{Goemans2009}.
In the following, we present two applications---the probabilistic covering location and multiple probabilistic knapsack problems---in detail.

\subsection{Minimum probabilistic covering location problem}
Given a set of customers $I$ with weight $v_i$ for each $i\in I$ and a set of candidate facility locations $J$, the minimal covering location problem (\MCLP) attempts to locate a fixed number of ``undesirable'' facilities (such as waste treatment plants and nuclear power stations) while
minimizing the total weight of covered nodes \citep{Church1976, Church2022, Murray1998}.
\citet{Karatas2021} investigated a probabilistic variant of the \MCLP, the minimum probabilistic covering location problem (\cGCMCLP), which incorporates two key practical factors.
First, the customers are covered in a probabilistic behavior rather than a deterministic behavior (in the \MCLP). 
Second, in addition to determining the locations of open facilities, their types must also be determined.
Let $x_{js}\in \{0,1\}$ be a binary variable denoting whether a facility of type $s \in S$ is located at location $j \in J$. Then,  
the \cGCMCLP can be formulated as:
\begin{equation}\label{cGC-MCLP}\tag{MPCLP}
	\!\!\!\!\min_{x\in \{0,1\}^{|J|\times |S|}}\left\{\sum_{i\in I}v_i \left(1-\prod_{j\in J}\prod_{s\in S} (1-p_{ijs}x_{js}) \right) \, : \,\sum_{j\in J}\sum_{s\in S} c_sx_{js}\geq t,~\sum_{s\in S}x_{js}\leq 1, ~\forall~ j\in J\right\},
\end{equation}
where  $p_{ijs}\in [0,1]$ denotes the probability that customer $i$ is covered by a facility of type $s$ located at $j$;
$1-\prod_{j\in J}\prod_{s\in S}(1-p_{ijs}x_{js})$  denotes the probability that customer $i$ is covered by at least one facility \citep{Drezner1997, Berman2003, Drezner2008, Karatas2017}; 
the objective function minimizes the expected total weight of covered customers;
$c_{s}$ denotes the service capacity of a facility of type $s$;
constraint $\sum_{j\in J}\sum_{s\in S} c_sx_{js}\geq t$ requires that the total service capacity of the open facilities is larger than or equal to a predetermined threshold $t$;
and constraint $\sum_{s\in S}x_{js}\leq 1$ enforces that at most one type $s \in S$ of facility is opened at each location $j \in J$.

Let any customer $i\in I$ be given and  
assume that $p_{ijs}\in [0,1)$ holds for all $j\in J$ and $s\in S$.
From $x_{j s} \in \{0, 1\}$, we obtain $(1-p_{ijs}x_{js})=(1-p_{ijs})^{x_{js}}=\text{exp }([\ln(1-p_{ijs})]x_{js})$.
By introducing a continuous variable $w_i$ to represent the nonlinear term corresponding to customer $i$ in the objective function of problem \eqref{cGC-MCLP} (i.e., $w_i \geq -\prod_{j\in J}\prod_{s\in S} (1-p_{ijs}x_{js}) $), we obtain a substructure taking the form of $X_0$: 
\begin{equation}\label{substructureX0i}
	\!\!\!\! X^i_0=\left\{(w_i,x)\in \R \times\{0,1\}^{|J|\times|S|} \,:\, w_i\geq f\left(\sum_{j\in J}\sum_{s\in S}[-\ln (1-p_{ijs})]x_{js}\right),
	~\sum_{s\in S}x_{js}\leq 1, ~\forall~ j\in J~\right\},
\end{equation}
where $f(z):=-\text{exp }(-z)$ is a concave function and $-\ln (1-p_{ijs})\geq 0$ (as $p_{ijs}\in [0,1)$).
Note that if $p_{ijs}=1$ holds for some $(j,s)\in J \times S$, then the nonlinear constraint $w_i \geq -\prod_{j\in J}\prod_{s\in S} (1-p_{ijs}x_{js}) $ can be equivalently represented as $w_i\geq f\left(\sum\limits_{(j,s)\in J \times S~\text{with}~p_{ijs} < 1}[-\ln (1-p_{ijs})]x_{js}\right)$ and $w_i\geq x_{js}-1$ for $(j,s)\in J\times S$ with $p_{ijs}=1$ \citep{Karatas2021}, and therefore, we can still obtain a  substructure taking the form of $X_0$:
\begin{equation}\label{substructurebarX0i}
\begin{aligned}
	&\bar{X}^i_0=\left\{(w_i,x)\in \R \times\{0,1\}^{|J|\times|S|} \,:\, w_i\geq f\left(\sum\limits_{(j,s)\in J \times S~\text{with}~p_{ijs} < 1}[-\ln (1-p_{ijs})]x_{js}\right),\right.\\
	&\hspace{11cm}\left. \sum_{s\in S}x_{js}\leq 1, ~\forall~ j\in J~\right\}.
\end{aligned}
\end{equation}

\subsection{The multiple probabilistic knapsack problem with \GUB constraints}
Given a finite index set $N$, a probabilistic knapsack constraint on a binary vector $x\in\{0,1\}^{|N|}$ is defined as 
\begin{equation}\label{probconstraint}
	\mathbb{P}\left(\sum_{i\in N}\tilde{a}_ix_i\leq b\right)\geq \rho,
\end{equation}
where $\tilde{a}$ is an $|N|$-dimensional random vector, $b\in \R$, and $\rho\in (0,1)$ is a reliability level chosen by the decision-maker.
Such a constraint arises in diverse industrial applications including healthcare \citep{Deng2016, Wang2017, Zhang2018}, transportation \citep{Dinh2018}, and scheduling \citep{Cohen2019, Lu2021}.
Under the realistic assumptions that $\{\tilde{a}_i\}_{i \in N}$ are independent normally distributed with mean $a_i$ and variance $\sigma_i^2$ for each $i \in N$, and $\rho >0.5$ \citep{Goyal2010, Han2016, Joung2017, Joung2020}, 
the probabilistic knapsack constraint \eqref{probconstraint} can be equivalently reformulated as $\sum_{i\in N}a_ix_i+\Phi^{-1}(\rho)\sqrt{\sum_{i\in N}\sigma_{i}^2x_i^2}\leq b$ \citep{Boyd2004}.
Here $\Phi^{-1}$ is the quantile function of the standard normal cumulative distribution, and $\Phi^{-1}(\rho)>0$ (as $\rho>0.5$).
\citet{Atamturk2013} considered the multiple probabilistic knapsack problem with \GUB constraints, denoted as \MPKPG:
\begin{equation}\label{problemSOCP}\tag{\MPKPG}
	\max_{x\in \{0,1\}^{|N|}}\left\{\sum_{i\in N}c_ix_i\,:\,
	\sum_{i\in N}a_{im}x_i +\Phi^{-1}(\rho)\sqrt{\sum_{i\in N}
		\sigma_{im}^2x_i^2}\leq b_m, ~\forall~ m\in M,
    \sum_{i\in Q_k} x_i\leq 1, ~\forall~ k\in K \right\},
\end{equation}
where $c\in \R^{|N|}$, $a\in \R^{|N|\times |M|}_+$, $\sigma\in \R^{|N|\times |M|}_+$, $b\in \R^{|M|}_+$, and $\{Q_k\}_{k\in K}$ is a partition of $N$.
Define $f(z):=\Phi^{-1}(\rho)\sqrt{z}$, which is a concave function.
As $x_i=x_i^2$ holds for $x_i\in\{0,1\}$, for each $m\in M$, the nonlinear constraint  in problem \eqref{problemSOCP} can be rewritten as 
$\sum_{i\in N}a_{im}x_i +f(\sum_{i\in N}\sigma_{im}^2x_i)\leq b_m$.
By introducing a continuous variable $w_m\in \R$ to represent the right-hand side (i.e., $w_m = b_m$), we have $\sum_{i\in N}a_{im}x_i +f(\sum_{i\in N} \sigma_{im}^2x_i)\leq w_m$, thereby obtaining a substructure taking the form of $X$: 
\begin{equation}\label{substructureXm}
	X^m = \left\{ (w,x) \in \mathbb{R}\times\{0,1\}^{|N|} \, : \, \sum_{i\in N}a_{im}x_i +f\left(\sum_{i\in N}
	\sigma_{im}^2x_i \right)\leq w_m, ~\sum_{i\in Q_k} x_i\leq 1, ~\forall~ k\in K \right\}.
\end{equation}

\subsection{Relevant literature}
It is known that (i) the composition of a non-negative linear function with a concave function yields a submodular function \citep{Ahmed2011}; and (ii) the sum of a submodular and a linear function is also submodular. 
Therefore, $X_0$ and $X$ are both submodular sets \citep{Edmonds2003,Yu2023} that additionally involve GUB constraints. 
Here, a function $g:\{0,1\}^n\rightarrow \R$ is submodular if for any $x,y\in \{0,1\}^n$ with $x\leq y$ (where the ``$\leq$'' is component-wise) and any $i\in [n]$ with $x_i=y_i=0$, it follows that $g(x+\mathbf{e}^i)-g(x)\geq g(y+\mathbf{e}^i)-g(y)$,
where $\mathbf{e}^i\in \R^n$ is the $i$-th standard unit vector and the submodular set is defined as $Y := \left\{(w,x) \in \R \times \{0, 1\}^n \,:\, w \geq g(x)\right\}$.
In the seminal work, \citet{Edmonds2003} developed the so-called extended polymatroid inequalities (\EPIs) that provide a complete linear description of $\conv(Y)$ and can be separated in $\CO(n\log n)$ time.
These results have been widely employed in strengthening the continuous relaxations of and improving the solution efficiency of the branch-and-cut algorithm for the mixed integer linear/nonlinear programming formulations in which $Y$ arises as a substructure \citep{Atamturk2008, Atamturk2020, Kilincc2025, Ahmed2013, Kilincc2022, Xie2021}. 

\citet{Yu2017a} considered a cardinality-constrained submodular set where the submodular function takes the form $f(a^\top x)$ (with $f:\R\rightarrow \R$ concave and $a\in \R^n_+$), that is,
$Z=\{(w,x)\in \R\times \{0,1\}^n \,:\, w\geq f(a^\top x), ~\sum_{i=1}^n x_i\leq c\}$ (with $c \in \Z_+$),
and presented a linear description for $\conv(Z)$ when all coefficients $\{a_i\}_{i \in [n]}$ are identical. 
For the general case where $\{a_i\}_{i \in [n]}$ are arbitrary positive {values}, \citet{Yu2023} derived strong valid inequalities for $\conv(Z)$ using sequential lifting techniques \citep{Wolsey1976, Richard2011}.
\citet{Atamturk2009} considered the submodular knapsack set $Y \cap \{w = b\}$ with $b \in \R$, and developed strong valid inequalities lifted from cover inequalities.
\citet{Yu2025} considered a mixed integer extension of submodularity, known as diminishing returns (DR)-submodularity, and gave the convex hull description of the epigraph of a DR-submodular function under box and monotonicity constraints.
For the submodular $\leq$-set given by $Y_{\leq}:=\left\{(w,x) \in \R \times \{0, 1\}^n \,:\, w \leq g(x)\right\}$,
\citet{Wolsey1999} established a linear characterization using the  submodular inequalities. 
These inequalities were later strengthened for the special case $\left\{(w,x) \in \R \times \{0, 1\}^n \,:\, w \leq f(a^\top x)\right\}$ (with $f:\R\rightarrow \R$ concave and $a\in \R^n_+$) or its constrained variants using sequence-independent lifting techniques; see \citet{Ahmed2011, Yu2017b}, and \citet{Shi2022}.

As discussed, if the \GUB constraints are removed, both $X$ and $X_0$ become submodular sets.
Therefore, the \EPIs, which are valid for submodular set $Y$,   are also valid for $X$ and $X_0$ (and can provide linear descriptions of the two sets).
Although the \EPIs are strong facet-defining inequalities of $\conv(Y)$, they may be weak for $\conv(X_0)$ and $\conv(X)$, as the derivation does not consider the \GUB constraints; see \cref{section2} for an illustration.
On the other hand, previous polyhedral results on various mixed integer (linear) sets \citep{Nemhauser1994, Sherali1995, Gu1998, Wolsey1990, Angulo2015, Gokce2015} have shown that by using the information of \GUB constraints, much stronger valid inequalities can usually be derived.
Motivated by this, we aim to develop strong valid inequalities for the mixed integer nonlinear sets $X_0$ and $X$ by explicitly taking the \GUB constraints into consideration.
 
\subsection{Contributions}
The primary goal of this work is to conduct an in-depth study of strong valid inequalities for $\conv(X_0)$ and $\conv(X)$, and to employ the derived inequalities as cutting planes to improve the computational performance for solving the related optimization problems.
In particular,
\begin{itemize}
	\item We first derive a new class of facet-defining inequalities for $\conv(X_0)$ by sequentially lifting a seed inequality $w\geq f(0)$ for a one-dimensional restriction of $X_0$, obtained by fixing all binary variables in $X_0$ to their lower bounds, according to some prespecified lifting sequence (i.e., a permutation of $[n]$).
	We derive closed-form formulae for the lifting coefficients, with which
	each lifted inequality can be computed in $\CO(n^2)$ time.
	We prove that for a fixed permutation of $[n]$, the lifted inequality, called lifted \EPI (\LEPI), is stronger than the classic \EPI \citep{Edmonds2003}, defined by the same permutation.
	\item We consider a subclass of permutations, which we refer to as \emph{partial ascending permutations} (each of which is a permutation $\D=(\D_1,\dots, \D_n)$ of $[n]$ such that for all $i,j\in [n]$ and $k\in K$ with $i < j $ and $\D_i,\D_j\in N_k$, it follows $a_{\D_i}\leq a_{\D_j}$), and study the corresponding \LEPIs.
	We show that (i) for a \LEPI  defined by an arbitrary permutation, there exists a partial ascending permutation that yields the same \LEPI;
	(ii) for any partial ascending permutation, the corresponding \LEPI can be computed in $\CO(n)$ time, as opposed to $\CO(n^2)$ time for the general case;
	and (iii) for any point $(w^*, x^*)\in \R \times [0,1]^n$ satisfying $\sum_{i\in N_k}x^*_i\leq 1$ for all $k\in K$, we can either verify that $(w^*, x^*)\in \conv(X_0)$ or find a violated \LEPI defined by some partial ascending permutation in $\CO(n\log n)$ time.
	The consequences of these results are two-fold.
	First, all \LEPIs (corresponding to the partial ascending permutations), along with bound and \GUB constraints, provide a complete linear description of $\conv(X_0)$, which
	demonstrates the strength of \LEPIs in strengthening the continuous relaxations of related problems for which $X_0$ appears as a substructure.  
	Second, the separation of \LEPIs (or $\conv(X_0)$) can be conducted in  $\CO(n\log n)$ time, which paves the way for the efficient generation of \LEPIs within a branch-and-cut framework.
	\item We extend our polyhedral results of $\conv(X_0)$ onto $\conv(X)$.
	In particular, we establish a one-to-one correspondence between the valid (and facet-defining) inequalities of $\conv(X_0)$ and those of $\conv(X)$.
	Therefore, $\conv(X)$ can be completely characterized using a class of variant \LEPIs (denoted as \LEPI's), which can also be separated in $\CO(n\log n)$ time.
\end{itemize}

We apply the proposed \LEPIs and \LEPI's as cutting planes within a branch-and-cut framework for solving the \eqref{cGC-MCLP} and \eqref{problemSOCP}, respectively.
Extensive computational experiments demonstrate that compared with the classic \EPIs, the proposed \LEPIs and \LEPI's are much more effective in strengthening the continuous relaxations and  improving the overall computational performance of the branch-and-cut framework for solving these two problems.

\subsection{Outline and notation}
The remainder of the paper is organized as follows.
In \cref{section2}, we present the classic \EPIs of \citet{Edmonds2003} and discuss their weakness for $\conv(X_0)$.
In \cref{section3}, we derive the facet-defining \LEPIs for $\conv(X_0)$ via sequential lifting techniques.
In \cref{section4}, we establish a more compact characterization of the \LEPIs using partial ascending permutations and show that each \LEPI defined by a partial ascending permutation can be computed in $\CO(n)$ time.
In \cref{section5}, we provide a complete linear description for $\conv(X_0)$ (using the \LEPIs corresponding to partial ascending permutations) and propose an $\CO(n\log n)$-time separation algorithm.
In \cref{section6}, we extend the above polyhedral results of $\conv(X_0)$ onto $\conv(X)$.
In \cref{section7}, we report the computational results. 
Finally, in \cref{section8}, we conclude the paper. 

We next present notation used throughout the paper.
For a permutation $\D=(\D_1,\D_2,\dots,\D_n)$ of $[n]$, we use $\D(j):=\{\D_1,\D_2,\dots,\D_j\}$ to denote the set involving the first $j$ items of $\D$ where $\D(0)=\varnothing$.
Given any $r_1,r_2\in [n]$ with $r_1=\D_{j_1}$ and $r_2=\D_{j_2}$, we call that (i) $r_1$ precedes $r_2$ in permutation $\D$ if $j_1<j_2$, and (ii) $r_1$ comes after $r_2$ in permutation $\D$ if $j_1>j_2$.
Let $\boldsymbol{e}_i$ and $\boldsymbol{0}$ be the $i$-th standard unit vector and the zero vector in $\R^n$, respectively. 
Given a vector $a\in \R^n$ and a set $S\subseteq [n]$, let $a(S):=\sum_{j\in S}a_j$ and $\chi^S:=\sum_{i\in S}\boldsymbol{e}_i$ denote the incidence vector of $S$.
We use $|S|$ to denote the cardinality of set $S$.
To simplify the notation, we abbreviate $S\cup \{j\}$ and $S \backslash \{j\}$ as $S\cup j$ and $S \backslash j$, respectively.
Without loss of generality, we impose the following assumptions on sets $X_0$ and $X$:
(i) $f(0)=0$ and $K=[t]$ for some $t\in \Z_{++}$;
(ii) the items in $N_k$ are indexed as follows:
\begin{equation}\label{assumption:Nk}
	N_k=\{i_{k-1}+1,i_{k-1}+2,\dots,i_{k-1}+|N_k|\}, ~\forall~k \in [t],
\end{equation}  
where $i_k:=\sum_{j=1}^{k}|N_j|$ and $i_0:=0$;
and (iii) $\{a_i\}_{i\in N_k}$ are arranged in ascending order:
\begin{equation}\label{asumption:a}
	a_{i_{k-1}+1}\leq a_{i_{k-1}+2}\leq\cdots\leq a_{i_{k-1}+|N_k|},~\forall~ k\in[t].
\end{equation}

\section{\EPIs and their weakness for conv$(X_0)$}\label{section2}
We first present the \EPIs \citep{Edmonds2003} for a generic submodular set: $Y=\{(w,x)\in\R\times\{0,1\}^n:w\geq g(x)\}$, 
where $g:\{0,1\}^n\rightarrow\R$ is a submodular function with $g(\boldsymbol{0})=0$.
Given any permutation $\D=(\D_1,\D_2,\dots,\D_n)$ of $[n]$, the corresponding \EPI takes the form 
\begin{equation}\label{EPI}\tag{EPI}
w\geq \sum_{j=1}^n\rho_{\D_j}x_{\D_j},
\end{equation}
where for each $j \in [n]$, $\D(j) = \{\D_1, \ldots, \D_j\}$ and $\rho_{\D_j}=g(\chi^{\D(j)})-g(\chi^{\D(j-1)})$.
Different permutations of $[n]$ may yield different \EPIs. 
From the classic results in \citet{Edmonds2003}, all \EPIs are facet-defining for $\conv(Y)$ and, along with the bound constraints $0\leq x_i\leq 1$ for $i\in [n]$, are able to describe $\conv(Y)$.
Moreover, \citet{Yu2023} showed that \EPIs can be derived using sequential lifting techniques, as formally stated in the following remark.
\begin{Remark}\label{remark:liftY}
	Let $w\geq g(\boldsymbol{0})=0$ be the seed inequality.
	Given any permutation $\D=(\D_1,\D_2,\dots,\D_n)$ of $[n]$, we can lift the seed inequality with the variables $\{x_i\}_{i\in [n]}$ in the order of $\D$ to obtain the lifted inequality
	\begin{equation}\label{liftY}
	w\geq\sum_{j=1}^n\xi_{\D_j}x_{\D_j},
	\end{equation}
	where the lifting coefficient $\xi_{\D_j}$ for $j\in [n]$ is computed by
	\begin{equation}\label{problem:liftY}
	\begin{aligned}
		\xi_{\D_j}:=\min_{w,\,x} \quad&w-\sum_{i=1}^{j-1}\xi_{\D_i}x_{\D_i}\\
		\textup{s.t.}\quad&w\geq g(x),\\
		&x_{\D_i}=0, \ \forall \ i\in \{j+1,\dots,n\},\\
		&x_{\D_j}=1,\\
		&x\in\{0,1\}^{n},~w\in\R.
	\end{aligned}
	\end{equation}
	\citet{Yu2023} showed that an optimal solution of problem \eqref{problem:liftY} is given by $(w,x)=(g(\chi^{\D(j)}), \chi^{\D(j)})$, and thus $\xi_{\D_j}=\rho_{\D_j}$ for all $j\in[n]$.
	As a result, the lifted inequality \eqref{liftY} coincides with inequality \eqref{EPI}.
\end{Remark}

Next, we apply Edmonds' result to the considered set $X_0$. 
As mentioned in \cref{section1}, $X_0$ is a submodular set with \GUB constraints:
\begin{equation*}
	X_0=Y\cap\left\{(w,x)\in\R\times\{0,1\}^n:\sum_{i\in N_k}x_i\leq 1,~\forall~k\in[t]\right\},
\end{equation*}
where the submodular function in $Y$ is $g(x) = f(a^\top x)$ (with $f \,:\, \R \rightarrow \R$ concave and $a \in \R_+^n$).
It immediately follows from Edmonds' result 
that (i) inequality \eqref{EPI} is valid for $X_0$ where $\rho_{\D_j}=f(a(\D(j)))-f(a(\D(j-1)))$ for $j\in [n]$ 
and (ii) all \EPIs, together with \GUB and binary constraints, provide a linear description of $X_0$.

Unfortunately, the derivation of \EPIs does not take the \GUB constraints $\{  \sum_{i \in N_k} x_i \leq 1\}_{k \in [t]}$ into account, which may lead to weak inequalities for $\conv(X_0)$.
We provide the following example to demonstrate this.
\begin{Example}\label{Example1}
	Letting $f(z)=-z^2$, $n=3$, $N_1=\{1,2\}$, $N_2=\{3\}$, and $a=(1,2,3)^\top$, then $X_0$ reduces to
	\begin{equation}\label{example1X_0}
		X_0=\left\{
		(w,x)\in \R \times\{0,1\}^3
		\,:\,
		w \geq -(x_1 + 2x_2 + 3x_3)^2, ~x_1 + x_2 \leq 1\right\}.
	\end{equation}
	All \EPIs for $X_0$ are presented as follows:
	\vspace{8pt}
	
	\begin{comment}
	\begin{align}
		&w+x_1+8x_2+27x_3\geq 0,\label{eq1}\\
		&w+x_1+20x_2+15x_3\geq 0,\label{eq2}\\\
		&w+5x_1+4x_2+27x_3\geq 0,\label{eq3}\\\
		&w+11x_1+4x_2+21x_3\geq 0,\label{eq4}\\\
		&w+7x_1+20x_2+9x_3\geq 0,\label{eq5}\\\
		&w+11x_1+16x_2+9x_3\geq 0.\label{eq6}\
	\end{align}
	\end{comment}
	\begin{minipage}{0.48\textwidth}
	\begin{align}
			&w+x_1+8x_2+27x_3\geq 0,\label{eq1}\\
			&w+x_1+20x_2+15x_3\geq 0,\label{eq2}\\
			&w+5x_1+4x_2+27x_3\geq 0,\label{eq3}\
	\end{align}
	\vspace{1pt}
	
	\end{minipage}
	\begin{minipage}{0.48\textwidth}
	\begin{align}
		&w+11x_1+4x_2+21x_3\geq 0,\label{eq4}\\
		&w+7x_1+20x_2+9x_3\geq 0,\label{eq5}\\
		&w+11x_1+16x_2+9x_3\geq 0.\label{eq6}\
	\end{align}
	\vspace{1pt}
	
	\end{minipage}
	Using \polymake \citep{Assarf2017}, we can obtain a linear characterization for $\conv(X_0)$:
	\vspace{8pt}
	
	\begin{comment}
	\begin{align}
		&w+x_1+4x_2+21x_3\geq 0,\label{eq7}\\
		&w+x_1+10x_2+15x_3\geq 0,\label{eq8}\\
		&w+7x_1+16x_2+9x_3\geq 0,\label{eq9}\\
		&x_1+x_2\leq 1,~x_3\leq 1,\nonumber\\
		&x_i\geq 0,~\forall~i\in[3].\nonumber
	\end{align}
	\end{comment}
	\begin{minipage}{0.48\textwidth}
	\begin{align}
			&w+x_1+4x_2+21x_3\geq 0,\label{eq7}\\
			&w+x_1+10x_2+15x_3\geq 0,\label{eq8}\\
			&w+7x_1+16x_2+9x_3\geq 0,\label{eq9}
	\end{align}
	\vspace{1pt}
	
	\end{minipage}
	\begin{minipage}{0.48\textwidth}
	\begin{align}
	    & \hspace{-1.2cm} x_1+x_2\leq 1,~x_3\leq 1,\\
	    & \hspace{-1.2cm} x_i\geq 0,~\forall~i\in[3].
	\end{align}
	\vspace{1cm}
	
	\end{minipage}
	None of the inequalities in \eqref{eq1}--\eqref{eq6} is facet-defining for $\conv(X_0)$.
	Indeed, inequalities \eqref{eq1}, \eqref{eq3}, and \eqref{eq4} are dominated by inequality \eqref{eq7}; inequality \eqref{eq2} is dominated by inequality \eqref{eq8}; and inequalities \eqref{eq5} and \eqref{eq6} are dominated by inequality \eqref{eq9}. 
	This example clearly shows that different from $\conv(Y)$ for which \EPIs are facet-defining,  for $\conv(X_0)$ (a restriction of $\conv(Y)$), \EPIs could be weak.
\end{Example}
This example motivates us to develop strong valid inequalities by directly investigating $\conv(X_0)$, rather than investigating its relaxation $\conv(Y)$.

\section{Strong valid inequalities for conv$(X_0)$}\label{section3}
In this section, we derive strong valid inequalities for $\conv(X_0)$ using sequential lifting techniques.
The lifting problem for $\conv(X_0)$ incorporates the \GUB constraints, and thus is more complicated than the lifting problem \eqref{problem:liftY} for $\conv(Y)$.
Nevertheless, we can still derive a closed-form optimal solution of the lifting problem for $\conv(X_0)$, rendering it efficient to compute the lifted inequality, called lifted \EPI (\LEPI).
We show that the \LEPIs are stronger than the classic \EPIs of \citet{Edmonds2003}.  

\subsection{Lifted extended polymatroid inequalities}\label{section3.1}
We first use sequential lifting techniques to derive strong valid inequalities for $\conv(X_0)$. 
For $S \subseteq [n]$, let 
\begin{equation*}
	X_0(S)=\left\{(w,x)\in\R\times\{0,1\}^{|S|} :w\geq f\left(\sum_{i\in S}a_ix_i\right),
	\sum_{i\in N_k\cap S}x_i\leq 1,  ~\forall~ k\in [t] \right\},
\end{equation*}
which is obtained by fixing $x_i$ to $0$ for all $i\in [n] \backslash S$ in $X_0$.
Note that the \emph{seed inequality} $w\geq f(0)=0$ defines a facet of $\conv(X_0(\varnothing))$.
To obtain a valid inequality
\begin{equation}\label{LEPI}\tag{LEPI}
	w\geq \sum_{i=1}^n\eta^{\D}_{\D_i}x_{\D_i},
\end{equation}
for $X_0$, 
we lift the \emph{seed inequality} with variables $\{x_i\}_{i\in [n]}$ according to some prespecified order/permutation $\delta=(\delta_1,\delta_2,\dots,\delta_n)$ (where $\{ \delta_1,\delta_2,\dots,\delta_n\} = [n]$).

In particular, in an intermediate step, a valid inequality $w\geq \sum_{i=1}^{j-1}\eta^{\D}_{\D_i}x_{\D_i}$ for $\conv(X_0(\D(j-1)))$
is derived and we attempt to lift variable $x_{\D_j}$, where
the lifting coefficient $\eta^{\D}_{\D_j}$ can be computed by solving the lifting problem:
\begin{subequations}\label{prob:liftX_0}
	\begin{align}
		\eta^{\D}_{\D_j}:=\min_{w,\,x} \quad &w-\sum_{i=1}^{j-1}\eta^{\D}_{\D_i}x_{\D_i}\\
		\text{s.t.} \quad &w\geq f\left(\sum_{i=1}^{j}a_{\D_i}x_{\D_i}\right),\label{prob:liftX_0-2}\\
		&\sum_{i\in N_k\cap\D(j)}x_i\leq 1, ~\forall~ k\in[t],\label{prob:liftX_0-3}\\
		&x_{\D_j}=1,~x\in\{0,1\}^{|\D(j)|}, ~ w \in\R.\label{prob:liftX_0-5}
	\end{align}
\end{subequations}
Since the \emph{seed inequality} $w\geq 0$ defines a facet of $\conv(X_0(\varnothing))$, it follows from the classic result of \citet{Wolsey1976} that \eqref{LEPI} defines a facet of $\conv(X_0)$.
\begin{Lemma}[cf. \citet{Wolsey1976}]\label{lem:facet}
	For any $j\in[n]$, $w\geq\sum_{i=1}^j\eta^{\D}_{\D_i}x_{\D_i}$ defines a facet of $\conv(X_0(\D(j)))$.
	In particular, \eqref{LEPI} defines a facet of $\conv(X_0)$.
\end{Lemma}

\subsection{Solving the lifting problems \eqref{prob:liftX_0}}\label{section3.2}
Next, we solve the lifting problems \eqref{prob:liftX_0} for  $j\in [n]$ to determine the lifting coefficients $\{\eta_{\D_j}^\D\}_{j \in [n]}$ and inequality \eqref{LEPI}.
Observe that when relaxing the complicated GUB constraints \eqref{prob:liftX_0-3} from problem \eqref{prob:liftX_0}, we obtain problem \eqref{problem:liftY}.
As stated in \cref{remark:liftY}, we can derive a closed-form optimal solution for this relaxation problem \eqref{problem:liftY}.
In the following, we show that even with the GUB constraints \eqref{prob:liftX_0-3}, a closed-form optimal solution for problem \eqref{prob:liftX_0} can still be derived.

We first transform the lifting problem \eqref{prob:liftX_0} into a set optimization problem. 
Note that for any optimal solution $(w,x) $ of problem \eqref{prob:liftX_0}, constraint \eqref{prob:liftX_0-2} holds at equality, i.e., $w=f\left(\sum_{i=1}^{j}a_{\D_i}x_{\D_i}\right)$.
Thus, we can project variable $w$ out from problem \eqref{prob:liftX_0}.
Given a feasible solution $x$ of problem \eqref{prob:liftX_0}, we denote the support of $x$ by $S = \{i \in [n] \,:\, x_i=1\}$.
Then the objective value of \eqref{prob:liftX_0} at $x$ is given by $f(a(S))-\sum_{i\in S \backslash\D_j}\eta^{\D}_i$.
Since $x$ satisfies constraints \eqref{prob:liftX_0-3} and \eqref{prob:liftX_0-5}, it follows that
$|S \cap N_k| \leq 1$ for $k \in [t]$, $\D_j \in S$, and $S \subseteq \D(j)$.
Therefore, the lifting problem \eqref{prob:liftX_0} is equivalent to
\begin{equation}\label{prob:equlift}
	\eta^\D_{\D_j} = \min_S \left\{  f(a(S))-\sum_{i\in S \backslash\D_j}\eta^{\D}_i \,:\,
	S \subseteq \D(j), ~ |S \cap N_k |\leq 1, ~\forall~k\in[t], ~\D_j \in S
	 \right\}.
\end{equation}
In what follows, we also refer to problem \eqref{prob:equlift} as the lifting problem.

Before deriving an optimal solution of problem \eqref{prob:equlift}, we introduce some notation. 
Define
\begin{equation}\label{def:U}
	U^{\D}_j : = \bigcup_{k=1}^t\left\{\max_{i\in N_k\cap\D(j)}i\right\}, ~ \forall~ j\in [n] \cup 0.
\end{equation}
Here and throughout the paper,  we let $\max_{i\in\varnothing}i=0$ and $\{\max_{i\in\varnothing}i\}=\varnothing$ (by abuse of notation).
Note that since $\delta(0) =\varnothing$, it follows that $U_0^\delta=\varnothing$. 
By definition and  \eqref{asumption:a}, it follows immediately that
\begin{Remark}\label{lem:propertyU}
	For any $j \in [n]$, (i) $|U_j^\delta \cup N_k|\leq 1$ holds for all $k \in [t]$; and 
	(ii) $a(U^{\D}_j)\geq a(S)$ holds for all $S \subseteq \D(j)$ with $|S \cap N_k |\leq 1$ for $k\in[t]$.
\end{Remark}
For $j \in [n]$, let $k\in [t]$ be such that  $\D_j \in N_k$  and 
\begin{equation}\label{problem:defh}
	h^{\D}_j:=
	\begin{cases}
		j, & 
		~\text{if}~ \D_j= \max_{i\in N_k\cap \D(j)} i,\\
	    \min\{\ell\in [j-1]\,:\,\D_{\ell}>\D_j,
	    ~\D_{\ell}\in N_k\}, 
	    & ~\text{if}~ \D_j < \max_{i\in N_k \cap \D(j)} i.
	\end{cases}
\end{equation}
Observe that if $\D_j= \max_{i\in N_k\cap \D(j)} i$, then by definition, we have $\D_{h^{\D}_j}=\D_j\in U^\D_j$;
otherwise, there exists some $\ell\in [j-1]$ 
such that $\D_{\ell} \in N_k$ and $\D_{\ell}>\D_j$, and as a result, 
$h^\D_j$ is well-defined.
From the definition of $h^\D_j$, $\D_{\ell} \leq \D_{h^\D_j}$ holds for all $\ell\leq h^\D_j$ (or equivalently, $\D_\ell \in \D(h^\D_j)$) with $\D_{\ell}\in N_k$,
and thus $\D_{h^\D_j}=\max_{i\in N_k\cap \D(h^\D_j)} i$.
As a result, 
\begin{Lemma}\label{lem:propertyh}
	$\D_{h^\D_j}\in U^\D_{h^\D_j}$ holds for all $j\in[n]$.
\end{Lemma}

Using the notation $U^\D_j$ and $h^\D_j$, we can provide a closed-form optimal solution for the lifting problem \eqref{prob:equlift}.
\begin{Theorem}\label{thm:lifting}
	Let $\D$ be any permutation of $[n]$.
	Then for $j\in[n]$, an optimal solution for the lifting problem \eqref{prob:equlift} is given by
	\begin{equation}\label{def:W}
		W^\D_j:=U^\D_{ h^\D_j}\backslash \D_{h^\D_j}\cup \D_j
	\end{equation}
	and the optimal value is
	\begin{equation*}
	\eta^{\D}_{\D_j}=f(a(W_j^\D))-\sum_{i\in W_j^\D \backslash\D_j}\eta^{\D}_i.
	\end{equation*}
\end{Theorem}

We proceed with the proof in the remainder of this subsection.
We first show that $W^\D_j$ is a feasible solution for the lifting problem \eqref{prob:equlift}.
\begin{Lemma}\label{lem:feasibilityW}
	For $j\in[n]$, $W_j^\D$ is a feasible solution of  problem \eqref{prob:equlift}.
\end{Lemma}
\begin{proof}
	It suffices to show that $\D_j\in W^\D_j$, $W^\D_j\subseteq \D(j)$, and $|W^\D_j \cap N_k |\leq 1$ for all $k \in [t]$.
	$\D_j\in W^\D_j$ directly follows from the definition of $W_j^\D$ in \eqref{def:W}.
	From the definition of $h^\D_j$ in \eqref{problem:defh}, it follows  $h^\D_j\leq j$, which, together with 
	 $U^\D_{h^\D_j}\subseteq \D(h^\D_j)$, implies $U^\D_{h^\D_j}\subseteq \D(j)$ and thus $W^{\D}_j\subseteq \D(j)$.
	 Finally, for any $k \in [t]$, $|W^\D_j \cap N_k |\leq 1$ follows from $|U^\D_{h^\D_j} \cap N_k |\leq 1$ (by \cref{lem:propertyU}(i)), $\D_{h^\D_j} \in U_{h^\D_j}^\D$, and $\D_{h^\D_j},\D_j\in N_{k'}$ for some $k'\in [t]$. 
\end{proof}
Next, we establish the optimality of $W^\D_j$ for problem \eqref{prob:equlift} by induction on $j$.
We begin with the base case $j=1$.
\begin{Lemma}\label{lem:case j=1}
	If $j =1$, then $W_1^\D$ is an optimal solution of problem \eqref{prob:equlift} and $\eta_{\D_1}^\D= f(a_{\D_1})$.
\end{Lemma}
\begin{proof}
	If $j=1$,  $W_1^\D= \{\D_1\}$ is the only feasible solution of problem \eqref{prob:equlift} and therefore is optimal.
\end{proof}
Now, suppose that $W^{\D}_j$ is an optimal solution of the lifting problem \eqref{prob:equlift} for all $j\in[r-1]$ where $r \in \{2,\ldots,n\}$; that is, 
\begin{equation}\label{induction hypothesis 1}
	\eta^{\D}_{\D_j}=f(a(W_j^\D))-\sum_{i\in W_j^\D \backslash\D_j}\eta^{\D}_i, ~\forall~ j\in[r-1].
\end{equation}
To show that $W^{\D}_r$ is an optimal solution of problem \eqref{prob:equlift} with $j=r$, we require the following two lemmas.
\begin{Lemma}\label{lem:validity}
	For any set $S \subseteq [n]$ satisfying $|S \cap N_k |\leq 1$ for $k\in[t]$, the inequality $f(a(S))\geq \sum_{i\in S}\eta^\D_i$ holds.
\end{Lemma}
\begin{proof}
	This follows directly from the validity of \eqref{LEPI} for $X_0$ and the observation that for any such set $S$, the point $(w,x)= (f(a(S)), \chi^S)$ belongs to $X_0$.
\end{proof}
\begin{Lemma}\label{lem:propertyeta}
	Suppose that \eqref{induction hypothesis 1} holds. Then it follows that
	\begin{equation}\label{property:eta}
		f(a(U^\D_j))=\sum_{i\in U^{\D}_j}\eta^{\D}_i, ~\forall~ j\in[r-1].
	\end{equation}
\end{Lemma}
\begin{proof}
	We prove the statement by induction on $j$.
	If $j=1$, then $U_1^\D= \{\D_1\}$, and by \cref{lem:case j=1}, it follows
	\begin{equation*}
		\sum_{i\in U^\D_1}\eta^\D_i=\eta^{\D}_{\D_1} =f(a_{\D_1})= f(a(U^\D_1)).
	\end{equation*}
	Suppose that the statement holds for $j=1, \ldots, \tau-1$ where $\tau$ is an integer satisfying $2 \leq \tau \leq r-1$.
	Now we consider the case $j=\tau$.
	Let $k' \in [t]$ be such that $\D_\tau\in N_{k'}$.
	We consider the following two cases.\\[8pt]
	(i) $\D_\tau=\max_{i\in N_{k'}\cap\D(\tau)}i$. Then, by the definition of $h^\D_\tau$ in \eqref{problem:defh} and the definition of $W^\D_\tau$ in \eqref{def:W}, we have $\D_\tau \in W^\D_\tau=U^\D_\tau$.
	Therefore, by \eqref{induction hypothesis 1}, it follows
	$\eta^{\D}_{\D_\tau}=f(a(U^\D_\tau))-\sum_{i\in U^{\D}_\tau\backslash\D_\tau}\eta^{\D}_i$,
	and thus, \eqref{property:eta} holds for $j=\tau$ in this case.\\[8pt]
	(ii) $\D_\tau<\max_{i\in N_{k'}\cap\D(\tau)}i$.
	Then $\max_{i\in N_{k'}\cap\D(\tau)}i=\max_{i\in N_{k'}\cap\D(\tau-1)}i$.
	Given any $k\in[t]\backslash k'$, it follows from $\D_\tau\notin N_k$ 
	that $\max_{i\in N_{k}\cap\D(\tau)}i=\max_{i\in N_{k}\cap\D(\tau-1)}i$.
	Therefore,
	\begin{equation*}
		U^{\D}_\tau=\bigcup^t_{k=1}\left\{\max_{i\in N_k\cap\D(\tau)}i\right\}=\bigcup^t_{k=1}\left\{\max_{i\in N_k\cap\D(\tau-1)}i\right\}=U^{\D}_{\tau-1}.
	\end{equation*}
	As a result,
	\begin{equation*}
		f(a(U^\D_\tau)) = f(a(U^\D_{\tau-1})) 
		\stackrel{(a)}{=}\sum_{i\in U^{\D}_{\tau-1}}\eta^{\D}_i
		= \sum_{i\in U^{\D}_\tau}\eta^{\D}_i,
	\end{equation*}
	where (a) follows from the induction hypothesis applied to $j = \tau -1$. 
\end{proof}
We now complete the induction step by proving that $W^{\D}_r$ is an optimal solution of problem \eqref{prob:equlift} with $j=r$.
\begin{Lemma}\label{lem:induction}
	Suppose that the induction hypothesis \eqref{induction hypothesis 1} holds. Then $W^{\D}_r$ is an optimal solution of problem \eqref{prob:equlift} with $j=r$.
\end{Lemma}
\begin{proof}
	Letting $S$ be an arbitrary feasible solution of the lifting problem \eqref{prob:equlift} with $j=r$, then 
	$\D_r\in S$, $S\subseteq\D(r)$, and $|S \cap N_k |\leq 1$ for $k\in[t]$.
	By \cref{lem:feasibilityW}, it suffices to show $f(a(W^\D_r))-\sum_{i\in W^{\D}_r\backslash \D_r} \eta^\D_i \leq f(a(S))-\sum_{i\in S\backslash \D_r}\eta^\D_i$, 
	or equivalently, 
	\begin{equation}\label{ourgoal2}
		 f(a(S))-f(a(W^\D_r)) \geq \sum_{i\in S\backslash \D_r}\eta^\D_i-\sum_{i\in W^{\D}_r\backslash \D_r} \eta^\D_i.
	\end{equation}
	Let $k' \in [t]$ be such that  $\D_r\in N_{k'}$.
    We next consider the following two cases (i) $a(W^{\D}_r)<a(S)$ and (ii) $a(W^{\D}_r)\geq a(S)$, separately. \\[8pt]
	{(i)} $a(W^{\D}_r)<a(S)$.
 	If $h^\D_r=r$, then $W^\D_r=U^\D_r$, which, 
	together with \cref{lem:propertyU}(ii), implies $a(S)\leq a(U^{\D}_r)=a(W^{\D}_r)$,
	a contradiction.
	Thus, by the definition of $h^\D_r$ in \eqref{problem:defh}, $h^\D_r< r$ and $\D_{h^\D_r}> \D_{r}$ must hold.
	Observe that
	\begin{equation}\label{tmpeq1}
		\begin{aligned}
			 &\sum_{i\in S\backslash\D_r}
			\eta^{\D}_i-\sum_{i\in W^{\D}_{r}\backslash\D_{r}}\eta^{\D}_i
			 \stackrel{(a)}{=} \sum_{i\in S\backslash \D_r}
			\eta^{\D}_i-\sum_{i\in U^{\D}_{h^\D_r}\backslash\D_{h^\D_r}}\eta^{\D}_i \\
			&\qquad \stackrel{(b)}{=} \sum_{i\in S\backslash \D_r\cup\D_{h^\D_r}}
			\eta^{\D}_i-\sum_{i\in U^{\D}_{h^\D_r}}\eta^{\D}_i 
			\stackrel{(c)}{=}  \sum_{i\in S\backslash \D_r\cup\D_{h^\D_r}}\eta^{\D}_i-f(a(U^{\D}_{h^\D_r})),
		\end{aligned}
	\end{equation}
	where (a) follows from the definition of $W^\D_j$ in \eqref{def:W} and $\D_r\notin U^\D_{h^\D_r}\backslash \D_{h^\D_r}$;
	(b) follows from $\D_{h^\D_r}\notin S\backslash \D_r$ (by $\D_r,\D_{h^\D_r}\in N_{k'}$, $|S \cap N_{k'}|\leq 1$, and $\D_r \in S$) and $\D_{h^\D_r}\in U^{\D}_{h^\D_r}$ (by \cref{lem:propertyh});
	and (c) follows from the induction hypothesis in \eqref{induction hypothesis 1} and \cref{lem:propertyeta}; that is, \eqref{property:eta} with $j=h^\D_r<r$ holds.
	 On the other hand,
	 \begin{equation}\label{tmpeq2}
	 	\begin{aligned}
	 		&f(a(S))-f(a(W^{\D}_r)) \stackrel{(a)}{=} f(a(S\backslash \D_r)+a_{\D_r})
	 		-f(a(U^{\D}_{h^\D_r}\backslash \D_{h^\D_r})+a_{\D_r})\\
	 		&\qquad \stackrel{(b)}{\geq} f(a(S\backslash \D_r)+a_{\D_{h^\D_r}})
	 		-f(a(U^{\D}_{h^\D_r}\backslash \D_{h^\D_r})+a_{\D_{h^\D_r}}) \\
	 		&\qquad \stackrel{(c)}{=}f(a(S\backslash \D_r\cup\D_{h^\D_r}))
	 		-f(a(U^{\D}_{h^\D_r}))\stackrel{(d)}{\geq} \sum_{i\in S\backslash \D_r\cup\D_{h^\D_r}}
	 		\eta^{\D}_i-f(a(U^{\D}_{h^\D_r})),
	 	\end{aligned}
	 \end{equation}
	 where (a) follows from $\D_r\in S$ and $\D_r\notin U^\D_{h^\D_r}\backslash \D_{h^\D_r}$; 
	 (b) follows from  $a(S\backslash \D_r)= a(S)-a_{\D_r}> 
	 a(W^{\D}_r)-a_{\D_r}= a(U^{\D}_{h^\D_r}\backslash\D_{h^\D_r})$,
	 $a_{\D_{h^\D_r}}\geq a_{\D_r}$ (by $\D_{h^\D_r},\D_r\in N_{k'}$, $\D_{h^\D_r}> \D_{r}$, 
	 and \eqref{asumption:a}), and the concavity of $f$ (that is, $f(y_2)-f(y_1)\geq f(y_2+d)-f(y_1+d)$ for any $d\in\R_+$ and $y_1,y_2\in\R$ with $y_1\leq y_2$);
	 (c) follows from $\D_{h^\D_r}\notin S\backslash \D_r$ 
	 (by $\D_r,\D_{h^\D_r}\in N_{k'}$ and $|S \cap N_{k'}|\leq 1$)
	 and $\D_{h^\D_r}\in U^{\D}_{h^\D_r}$ (by \cref{lem:propertyh});
	 and (d) follows from \cref{lem:validity} and $|(S\backslash \D_r\cup \D_{h^\D_r})\cap N_k|\leq 1$ for $k \in [t]$.
	 Combining \eqref{tmpeq1} and \eqref{tmpeq2} yields the desired result in \eqref{ourgoal2}.\\[8pt]
	{(ii)} $a(W^{\D}_r)\geq a(S)$.
	Then, by $W_r^\D= U^\D_{ h^\D_r}\backslash \D_{h^\D_r}\cup \D_r$, $\D_r \notin U^\D_{ h^\D_r}\backslash \D_{h^\D_r}$, and $\D_r \in S$, it follows
	\begin{equation}\label{tmpeq3}
		a(U^\D_{h^\D_r}\backslash\D_{h^\D_r})
		= a(W^\D_r)-a_{\D_r}\geq a(S)-a_{\D_r}= a(S\backslash \D_r).
	\end{equation}
	Let
	\begin{equation}\label{defU'}
		U':=
		\begin{cases}
			\varnothing & 
			~\text{if}~  N_{k'}\cap\D(h^\D_r-1)=\varnothing,\\
			\{\max_{i\in N_{k'}\cap\D(h^\D_r-1)}i\},
			& ~\text{otherwise}.
		\end{cases}
	\end{equation}
	By definition, it follows that
	\begin{equation}\label{J'-1J'}
		U^\D_{h^\D_r-1}\backslash U'= \bigcup_{k\in [t]\backslash k'}
		\left\{\max_{i\in N_k\cap\D(h^\D_r-1)}i\right\}= \bigcup_{k\in [t]\backslash k'}
		\left\{\max_{i\in N_k\cap\D(h^\D_r)}i\right\}=U^\D_{h^\D_r}\backslash \D_{h^\D_r}
	\end{equation}
	and hence
	\begin{equation}\label{tmpeq4}
		a(U^{\D}_{h^\D_r-1})= a(U^{\D}_{h^\D_r}\backslash\D_{h^\D_r})+a(U').
	\end{equation}
	From the definition of $h^\D_r$ in \eqref{problem:defh}, 
	$\D_{\ell}<\D_r$ holds for all $\ell<h^\D_r$ with $\D_{\ell}\in N_{k'}$.
	This, together with \eqref{asumption:a} and the definition of $U'$ in \eqref{defU'}, implies $a(U')\leq a_{\D_r}$.
	To prove the desired result in \eqref{ourgoal2}, observe that
	\begin{equation*}
		\begin{aligned}
			&f(a(S))-f(a(W^{\D}_r))=-[f(a(U^{\D}_{h^\D_r}
			\backslash \D_{h^\D_r})+a_{\D_r})-f(a(S\backslash \D_r)+a_{\D_r})]\\
			&\qquad \stackrel{(a)}{\geq} -[f(a(U^\D_{h^\D_r}\backslash \D_{h^\D_r})+a(U'))-f(a(S\backslash \D_r)+a(U'))]
			\stackrel{(b)}{=}f(a(S\backslash \D_r\cup U'))-f(a(U^{\D}_{h^\D_r-1}))\\
			&\qquad \stackrel{(c)}{\geq} \sum_{i\in S\backslash \D_r\cup U'}\eta^{\D}_i-f(a(U^{\D}_{h^\D_r-1}))
			\stackrel{(d)}{=}\sum_{i\in S\backslash \D_r\cup U'}\eta^{\D}_i- \sum_{i\in U^{\D}_{h^\D_r-1}}\eta^{\D}_i\\
			&\qquad \stackrel{(e)}{=} 
			\sum_{i\in S\backslash \D_r}\eta^{\D}_i-\sum_{i\in U^{\D}_{h^\D_r-1}\backslash U'}\eta^{\D}_i
			\stackrel{(f)}{=} \sum_{i\in S\backslash\D_r}\eta^{\D}_i-\sum_{i\in W^{\D}_{r}\backslash\D_{r}}
			\eta^{\D}_i,
		\end{aligned}
	\end{equation*} 
	where (a) follows from \eqref{tmpeq3}, $a(U')\leq a_{\D_r}$, and the concavity of $f$;
	(b) follows from \eqref{tmpeq4};
	(c) follows from \cref{lem:validity} and $|(S\backslash \D_r\cup U')\cap N_k|\leq 1$ for $k \in [t]$;
	(d) follows from the induction hypothesis in \eqref{induction hypothesis 1} and \cref{lem:propertyeta}, that is, \eqref{property:eta} with $j=h^\D_r-1\leq r-1$ holds;
	(e) follows from $U'\cap (S\backslash \D_r)=\varnothing$ and $U'\subseteq U^{\D}_{h^\D_r-1}$; and 
	(f) follows from \eqref{J'-1J'} and $U^\D_{h^\D_r}\backslash \D_{h^\D_r}=W^\D_r\backslash \D_r$. 
\end{proof}
\begin{proof}{(Proof of \cref{thm:lifting})}
	Combining Lemmas \ref{lem:case j=1} and \ref{lem:induction}, we obtain the desired result in \cref{thm:lifting}.
\end{proof}

Given any permutation $\D$ of $[n]$, it follows from \eqref{problem:defh} and \cref{thm:lifting} that
\begin{Proposition}\label{prop:complex}
 The lifting coefficients $\{\eta^\D_{\D_j}\}_{j\in [n]}$ can be computed in $O(n^2)$ time.
\end{Proposition}

\subsection{Strength of \LEPIs over \EPIs}\label{section3.3}
In this subsection, we show that  \LEPIs are stronger than \EPIs in the sense that for a fixed permutation $\D$ of $[n]$, \eqref{LEPI} is at least as strong as \eqref{EPI}.
\begin{Theorem}\label{thm:LEPI-EPI}
	Given any permutation $\D$ of $[n]$, it follows that $\eta^\D_{\D_j}\geq \rho_{\D_j}=f(a(\D(j)))-f(a(\D(j-1)))$ for $j\in [n]$.
\end{Theorem}
\begin{proof}
	Let any $j\in [n]$ be given.
	Note that for any set $S\subseteq [n]$ satisfying $|S\cap N_k|\leq 1$ for $k\in [t]$, it follows from \cref{lem:validity} that
	$f(a(S))-\sum_{i\in S\backslash\D_j}\eta^{\D}_i\geq f(a(S))-f(a(S\backslash\D_j))$.
	As a result,
	\begin{equation*}
		\begin{aligned}
			\eta^\D_{\D_j}=&\min_S \left\{  f(a(S))-\sum_{i\in S \backslash\D_j}\eta^{\D}_i \,:\,
			S \subseteq \D(j), ~|S\cap N_k|\leq 1, ~\forall~ k\in [t], ~\D_j \in S \right\},\\
			\geq& \min_S \left\{  f(a(S))-f(a(S\backslash\D_j)) \,:\,
			S \subseteq \D(j), ~|S\cap N_k|\leq 1, ~\forall~ k\in [t], ~\D_j \in S\right\},\\
			\geq & \min_S \left\{  f(a(S))-f(a(S\backslash\D_j)) \,:\,
			S \subseteq \D(j), ~\D_j \in S
			\right\},\\
			=&f(a(\D(j)))-f(a(\D(j-1))),
		\end{aligned}
	\end{equation*}
	where the last equality holds due to $a_{\D_j} \geq 0$, $a(\D(j)) \geq a(S)$ for all $S$ with $S \subseteq \D(j)$ and $\D_j \in S$ (since $a\in\R^n_+$),
	and the concavity of $f$. 
\end{proof}
We provide an example to illustrate \cref{thm:LEPI-EPI}.
\begin{Example}[\cref{Example1} Continued]\label{Example2}
	Consider $X_0$ in \eqref{example1X_0} of \cref{Example1}.
	We list \EPIs and \LEPIs corresponding to all the permutations of $[3]$:
	\begin{align*}
		&(1,2,3):\quad w+1x_1+8x_2+27x_3\geq 0,\quad w+1x_1+\boldsymbol{4}x_2+\boldsymbol{21}x_3\geq 0.\\
		&(1,3,2):\quad w+1x_1+20x_2+15x_3\geq 0,\quad w+1x_1+\boldsymbol{10}x_2+15x_3\geq 0.\\
		&(2,1,3):\quad w+5x_1+4x_2+27x_3\geq 0,\quad w+\boldsymbol{1}x_1+4x_2+\boldsymbol{21}x_3\geq 0.\\
		&(2,3,1):\quad w+11x_1+4x_2+21x_3\geq 0,\quad
		w+\boldsymbol{1}x_1+4x_2+21x_3\geq 0.\\
		&(3,1,2):\quad w+7x_1+20x_2+9x_3\geq 0,\quad w+7x_1+\boldsymbol{16}x_2+9x_3\geq 0.\\
		&(3,2,1):\quad w+11x_1+16x_2+9x_3\geq 0,\quad w+\boldsymbol{7}x_1+16x_2+9x_3\geq 0.
	\end{align*}
	Clearly, for any permutation $\D$ of $[3]$, \eqref{LEPI} defined by the same $\D$ dominates \eqref{EPI} defined by $\D$. This observation is consistent with \cref{thm:LEPI-EPI}.
\end{Example}

\cref{Example2} reveals two important properties of \LEPIs. 
First, different permutations of $[n]$ may define the same \eqref{LEPI}.
In the next section, we will identify \emph{partial ascending permutations} that are enough to characterize all \LEPIs.
Second, all \LEPIs, together with \GUB and bound constraints, provide a complete linear description of $\conv(X_0)$; see this description in \cref{Example1}.
In \cref{section5}, we will show that 
such property of \LEPIs is indeed not random, but hold in general.

\section{A more compact characterization of \LEPI{s}}\label{section4}

In this section, we will identify a subclass of permutations, called \emph{partial ascending permutations}, with which all \LEPIs can be characterized. 
We show that \eqref{LEPI} defined by any partial ascending permutation can be computed in $\mathcal{O}(n)$ time. 
\subsection{Characterizing \eqref{LEPI} by the set collection $\{W_j^\D\}_{j\in [n]}$}
Let $\D = (\D_1, \D_2, \ldots, \D_n)$ be a permutation of $[n]$ that defines \eqref{LEPI}.
Obviously, the point $(w,x)=(0, \boldsymbol{0})\in X_0$ satisfies \eqref{LEPI} at equality, and 
from \cref{thm:lifting}, the $n$ points $\{(f(a(W^\D_j)), \chi^{W^\D_j})\}_{j\in [n]}\in X_0$ also satisfy \eqref{LEPI} at equality, where we recall that $W_j^\D = U^\D_{h^\D_j} \backslash \D_{h^\D_j} \cup \D_j$ and $\chi^{W^\D_j} = \sum_{i \in W^\D_j} \boldsymbol{e}_i$ for $j\in [n]$.
Moreover, these $n+1$ points are affinely independent (as $W^\D_j$ is the only set in $\{W^\D_r\}_{r\in [j]}$ that contains element $\D_j$ for any $j\in [n]$), thereby determining the facet-defining inequality \eqref{LEPI}.
Since the points $\{(f(a(W^\D_j)), \chi^{W^\D_j})\}_{j\in [n]}$ are defined by the set collection  $\{W_j^\D\}_{j\in [n]}$, 
it follows that
\begin{Remark}\label{prop:set-collection}
	\eqref{LEPI} can be determined by the set collection $\{W_j^\D\}_{j\in [n]}$.
\end{Remark}
 Different permutations may lead to the same set collection $\{W_j^\D\}_{j\in [n]}$ and thus the same \eqref{LEPI}.
\begin{Example}[\cref{Example1} Continued]\label{Example3}
   	The set collections $\{W^\D_j\}_{j\in [3]}$ defined in \eqref{def:W} for all permutations $\D$ of $[3]$ are listed in \cref{table1}. 
   	As shown in \cref{table1}, 
   	permutations $(1,2,3)$, $(2,1,3)$, and $(2,3,1)$ define the same set collection  $\{\{1\},\{2\},\{2,3\}\}$, and thus  the corresponding \LEPIs are identical; see also \cref{Example2}.
   	Similarly, the \LEPIs defined by
   	permutations $(3,1,2)$ and $(3,2,1)$ are also identical.
    \begin{table}[htbp]
    	\centering
    	\fontsize{13}{15}\selectfont
    	\caption{The set collections $\{W^\D_j\}_{j\in [3]}$ for all permutations $\D$ of $[3]$ in \cref{Example1}.}
    	\label{table1}
    	\begin{tabular}{c c c c c c c c c c c}
    		\toprule
    		$\D$ & $h^\D_1$ & $U^\D_{h^\D_1} $ & $W^\D_1$
    		& $h^\D_2$ & $U^\D_{h^\D_2} $ & $W^\D_2$
    		& $h^\D_3$ & $U^\D_{h^\D_3} $ & $W^\D_3$ & $\{W^\D_j\}_{j\in[3]}$\\
    		\midrule
    		$(1, 2, 3) $ & 1 & $\{1\}$ & $\{1\}$
    		& 2 & $\{2\}$ & $\{2\}$
    		& 3 & $\{2, 3\}$ & $\{2, 3\}$ & $\{\{1\},\{2\},\{2,3\}\}$\\
    		$(2, 1, 3) $ & 1 & $\{2\}$ & $\{2\}$
    		& 1 & $\{2\}$ & $\{1\}$
    		& 3 & $\{2, 3\}$ & $\{2, 3\}$ & $\{\{1\},\{2\},\{2,3\}\}$\\
    		$(2,3,1)$ & 1 & $\{2\}$ & $\{2\}$
    		& 2 & $\{2,3\}$ & $\{2,3\}$
    		& 1 & $\{2\}$ & $\{1\}$ & $\{\{1\},\{2\},\{2,3\}\}$\\
    		$(3,1,2)$ & 1 & $\{3\}$ & $\{3\}$
    		& 2 & $\{1,3\}$ & $\{1,3\}$
    		& 3 & $\{2,3\}$ & $\{2,3\}$ & $\{\{3\},\{1,3\},\{2,3\}\}$\\
    		$(3,2,1)$ & 1 & $\{3\}$ & $\{3\}$
    		& 2 & $\{2,3\}$ & $\{2,3\}$
    		& 2 & $\{2,3\}$ & $\{1,3\}$ & $\{\{3\},\{1,3\},\{2,3\}\}$\\
    		$(1,3,2)$ & 1 & $\{1\}$ & $\{1\}$
    		& 2 & $\{1,3\}$ & $\{1,3\}$
    		& 3 & $\{2,3\}$ & $\{2,3\}$ & $\{\{1\},\{1,3\},\{2,3\}\}$\\
    		\bottomrule
    	\end{tabular}
    \end{table}
\end{Example}

\subsection{Partial ascending permutations}
To avoid the unnecessary permutations that define the same set collection $\{W_j^\D\}_{j\in [n]}$ (and thus the same \eqref{LEPI}),  we introduce the \emph{partial ascending permutations}, which maintain the ascending order of all items within each \GUB set. 
\begin{Definition}\label{Def:validperm}
	A permutation $\D$ of $[n]$ is called partial ascending if for all $i,j\in [n]$ and $k\in [t]$ with $i < j $ and $\D_i,\D_j\in N_k$, it follows $\D_i<\D_j$. 
\end{Definition}

In the remainder of this paper, let $\Delta$ denote the set of all partial ascending permutations of $[n]$.

\begin{Lemma}\label{prop:W-equals-U}
	The following four statements are equivalent: (i) $\D\in \Delta$; 
	(ii) $\D_j=\max_{i\in N_k\cap \D(j)}i$, or equivalently, $U_j^\D \cap N_k = \{\D_j\}$, holds for all $j \in [n]$, 
	where $k \in [t]$ satisfying $\D_j \in N_k$ and $U^{\D}_j$ is defined in \eqref{def:U}; 
	(iii) $j=h_j^\D$ holds for all $j \in [n]$ where  $h^\D_j$ is defined in \eqref{problem:defh};
	and (iv) $W^{\D}_j = U^{\D}_j$ holds for all $ j\in [n]$, where $W^{\D}_j$ is  defined in \eqref{def:W}. 
\end{Lemma}
\begin{proof}
	The equivalence of (i) and (ii) follows from the definition of partial ascending permutations and the fact that $\D(j) = \{\D_1, \ldots, \D_j\}$.
	The equivalence of (ii) and (iii) follows from  the definition of $h^\D_j$ in \eqref{problem:defh}.
	Finally, $W^\D_j = U^\D_{ h^\D_j}\backslash \D_{h^\D_j}\cup \D_j= U_j^\D$ holds if and only if $j=h_j^\D$, which implies the equivalence of (iii) and (iv). 
\end{proof}

Given any permutation $\D$, we can use an iterative procedure to construct a permutation $\D'\in \Delta$.
Specifically, in the $j$-th iteration, we check whether  $\D_j=\max_{i\in N_k\cap \D(j)}i$ holds (where $k\in [t]$ satisfying $\D_j\in N_k$); see \cref{prop:W-equals-U}(ii). 
If it does not hold, we move $\D_j$ to the position right before $\D_{h_j^\D}$ (where we recall that $\D_{h_j^\D}$ is the first item in $\D$ that is larger than $\D_j$ and belongs to $N_k$). 
The procedure is repeated until $j=n$. 
\cref{alg:permutationfinder} summarizes the overall procedure. 
Note that after the $j$-th iteration of \cref{alg:permutationfinder}, $\D^{j}_\tau = \max_{i\in N_k\cap \D^{j}(\tau)} i$ must hold for $\tau=1, \ldots, j$, and therefore, it follows from \cref{prop:W-equals-U}(ii) that  \cref{alg:permutationfinder} will output a permutation $\D'\in \Delta$.

\begin{algorithm}[htbp]
	\SetAlgoLined
	\textbf{Input} A permutation $\D$ of $[n]$\;
	$\D^0 \leftarrow \D$\; 
	\For{$j=1,2,\dots,n$}{
		Let $k \in [t]$ be such that $\D^{j-1}_j\in N_k$\;
		\eIf{$\D^{j-1}_j = \max_{i\in N_k\cap \D^{j-1}(j)} i$}{$\D^j\leftarrow \D^{j-1}$\;}{
			$j' \leftarrow h^{\D^{j-1}}_j$;	~~\tcp{$h^{\D^{j-1}}_j$ is the index of the first item in $\D^{j-1}$ that is larger than $\D^{j-1}_j$ and belongs to $N_k$.}
			$\D^j \leftarrow (\D^{j-1}_1,\dots,\D^{j-1}_{j'-1},\D^{j-1}_j,\D^{j-1}_{j'},\dots,\D^{j-1}_{j-1},\D^{j-1}_{j+1},\dots,\D^{j-1}_n)$; ~~\tcp{move $\D^{j-1}_j$ to the position right before $\D^{j-1}_{j'}$.}
		}
	}
	$\D' \leftarrow \D^n$\;
	\textbf{Output} A partial ascending permutation $\D'\in \Delta$.
	\caption{{An iterative procedure to construct a partial ascending permutation $\D'$ from a given permutation $\D$ of $[n]$}}
	\label{alg:permutationfinder}
\end{algorithm}

\cref{alg:permutationfinder} constructs $n$ permutations: $\D^1, \ldots, \D^n$. 
We now summarize some properties of these permutations as follows.
\begin{Remark}\label{ob:condition}
	Let $\D^1, \ldots, \D^n$ be the permutations constructed in \cref{alg:permutationfinder}.
	Then the following statements hold.
	\begin{itemize}
		\item [\rm{(i)}] 
		For any $j \in [n]$, it follows that $\D^j(j) = \D(j)$.
		\item [\rm{(ii)}] Letting  $\tau_1,  \tau_2, j \in [n] $ be such that $\tau_1 \leq \tau_2 \leq j$, then for any $r \geq j$	, $\D_{\tau_1}^j$ does not come after $\D_{\tau_2}^j $ in permutation $\D^r$.
		\item [\rm{(iii)}] For any $k \in [t]$ and $j \in [n]$, let $\beta = \max_{i \in N_k \cap \D^j(j)}i $. 
		Then for $r \in  [n]$ with $r >j$, it follows $\beta = \max_{i \in N_k \cap \D^r(j')}i  $ where $j'$ is the position of $\D_{j}^j$ in $\D^r$ (that is, $\D^r_{j'}=\D_j^j$).
	\end{itemize}
\end{Remark}

Using the above results, we can give a sufficient condition under which $\D_j$ precedes $\D_{\ell}$ in permutation $\D'$.
\begin{Lemma}\label{lem:precede}
	Let $\ell,j\in [n]$ be such that $\D_j$ precedes $\D_\ell$ in permutation $\D$ (i.e., $j < \ell$), and $k\in [t]$ be such that $\D_{\ell}\in N_k$.
	If $\D_{\ell}> \max_{i\in N_k\cap \D(j)}i$, then $\D_j$ precedes $\D_{\ell}$ in permutation $\D'$.
\end{Lemma}
\begin{proof}
	First, from \cref{ob:condition}(i), $\D_j \in \D(j)=\D^j(j)$ holds and thus $\D_j$ does not come after $\D^j_j$ in permutation $\D^j$.
	Then, using \cref{ob:condition}(ii), $\D_j$ must not come after $\D^j_j$ in permutations $\D^{j+1}, \D^{j+2}, \ldots, \D^n = \D'$.
	Finally, it follows from  \cref{ob:condition}(iii) and $\D_\ell> \max_{i \in N_k \cap \D(j)}i=\max_{i \in N_k\cap \D^j(j)}i$ that $\D_j^j$ precedes $\D_\ell$ in permutation $\D^n = \D'$. 
\end{proof}
The following theorem demonstrates that the set collection $\{W^\D_j\}_{j\in [n]}$ for permutation $\D$ is identical to that $\{W^{\D'}_j\}_{j\in [n]}$ for the partial ascending permutation $\D'$, where $\D' $ is computed by \cref{alg:permutationfinder}.
\begin{Theorem}\label{thm:D-D'}
	Let $\D$ be a permutation of $[n]$ and $\D'$ be the corresponding partial ascending permutation returned by \cref{alg:permutationfinder}. 
	Then $\{W^\D_j\}_{j\in [n]}=\{W^{\D'}_j\}_{j\in [n]}$.
\end{Theorem}
\begin{proof}
	From \cref{prop:W-equals-U}(iv) and the fact that $\D'\in \Delta$, it suffices to show 
	$\{W^\D_j\}_{j\in [n]}=\{U^{\D'}_j\}_{j\in [n]}$.
	Observe that for any $j\in [n]$,  
	there exists a unique $j'\in [n]$ such that $\D_j=\D'_{j'}$.
	In the following, we shall prove the statement by showing that
	$W^\D_j=U^{\D'}_{j'}$, or equivalently, 
	\begin{equation}\label{tmpeq}
		W^\D_j\cap N_k=U^{\D'}_{j'}\cap N_k, ~\forall~k\in [t].
	\end{equation}
	
	Let $k'\in [t]$ be such that $\D_j\in N_{k'}$.
	From the definition of $W_j^\D$ in \eqref{def:W}, we have $W_j^\D \cap N_{k'} = \{\D_j\}= \{\D'_{j'}\}$, which, together with \cref{prop:W-equals-U}(ii) and $\D'\in \Delta$, implies that $U^{\D'}_{j'}\cap N_{k'}=\{\D'_{j'}\}=W_j^\D \cap N_{k'} $.
	Therefore, \eqref{tmpeq} holds for $k= k'$. 
	Next, we consider the case $k\in [t]\backslash k'$.
	In this case, $W^\D_j \cap N_k = (U^\D_{h^\D_j} \backslash \D_{h^\D_j} \cup \D_j) \cap N_k =
	 U^\D_{h^\D_j}\cap N_k$, and thus it suffices to show that $U^\D_{h^\D_j}\cap N_k =U^{\D'}_{j'}\cap N_k$, or equivalently,
	\begin{equation}\label{eq:1}
		\max_{i\in N_k\cap \D(h^\D_j)}i=\max_{i\in N_k\cap \D'(j')}i.
	\end{equation}
	
	To prove \eqref{eq:1}, we first show that 
	\begin{equation}\label{eq:geq}
		\max_{i\in N_k\cap \D(h^\D_j)}i \geq \max_{i\in N_k\cap \D'(j')}i.
	\end{equation}
	Let $\tau\in [n]$ be such that $\D_{h^\D_j}=\D'_{\tau}$.
	Then $\D'_{\tau} \in N_{k'}$, which, together with $\D'_{j'} \in N_{k'}$,
	$\D'_{\tau}=\D_{h^\D_j}\geq \D_j=\D'_{j'}$, and $\D'\in \Delta$, implies that $\tau\geq j'$.
	As a result, 
	\begin{equation}\label{eq:geq1}
		\max_{i\in N_k\cap \D'(\tau)}i \geq \max_{i\in N_k\cap \D'(j')}i.
	\end{equation}	
	Let $\ell\in [n]$ be such that $\D_{\ell}\in N_k\cap \D'(\tau)$.
	If $\D_{\ell}>\max_{i\in N_k\cap \D(h^\D_j)}i$, then $\ell>h^\D_j$,
	and by \cref{lem:precede}, $\D_{h^\D_j}$ ($=\D'_{\tau}$) precedes $\D_{\ell}$ in permutation $\D'$. 
	Thus, $\D_{\ell}\notin \D'(\tau)$, a contradiction with $\D_{\ell}\in N_k\cap \D'(\tau)$.
	As a result, $\D_{\ell}\leq \max_{i\in N_k\cap \D(h^\D_j)}i$.
	Since $\ell$ is chosen arbitrarily from the set of indices such that $\ell\in [n]$ and $\D_{\ell}\in N_k\cap \D'(\tau)$,
	it follows that
	\begin{equation}\label{eq:geq2}
		\max_{i\in N_k\cap \D(h^\D_j)}i \geq \max_{i\in N_k\cap \D'(\tau)}i.
	\end{equation}
	Combining \eqref{eq:geq1} and \eqref{eq:geq2} yields \eqref{eq:geq}.
	
	Next, we show that 
	\begin{equation}\label{eq:leq}
		\max_{i\in N_k\cap \D(h^\D_j)}i \leq \max_{i\in N_k\cap \D'(j')}i.
	\end{equation}
	If $N_k\cap \D(h^\D_j)=\varnothing$, then by $\max_{i\in \varnothing} i = 0$, \eqref{eq:leq} holds.
	Otherwise, let $\D_s= \max_{i\in N_k\cap \D(h^\D_j)}i$.
	By definition,  it follows that $s \leq h^\D_j$ (as $\D_s \in \D(h^\D_j)$) and $\D_s \in N_k$.
	Together with $\D_{h^\D_j} \in N_{k'}$ and $k\neq k'$, we obtain $s \neq h^\D_j$, and hence $s \leq h^\D_j - 1$.
	As a result, $\max_{i\in N_{k'}\cap \D(s)}i\leq \max_{i\in N_{k'}\cap \D(h^\D_j-1)}i< \D_j$, where the last inequality follows from the definition of $h^\D_j$ in \eqref{problem:defh}.
	This, together with $s \leq h^\D_j - 1<h^\D_j \leq j$,  $\D_j\in N_{k'}$, and \cref{lem:precede}, implies that $\D_s$ precedes $\D_j (= \D'_{j'})$ in $\D'$ (i.e., $\D_s\in \D'(j')$).
	Therefore, $\D_s\in N_k\cap \D'(j')$, and hence
	$\max_{i\in N_k\cap \D'(j')}i\geq \D_{s}= \max_{i\in N_k\cap \D(h^\D_j)}i$. 
\end{proof}

From \cref{thm:D-D'}, for any permutation $\D$ of $[n]$ that is not partial ascending,
the corresponding  set collection $\{W^\D_j\}_{j \in [n]}$ is identical to that corresponds to a partial ascending permutation returned by \cref{alg:permutationfinder}.
This, together with the fact that \eqref{LEPI} can be determined by a set collection, immediately implies the following result.
\begin{Corollary}\label{coro:validD}
	Given a permutation $\D$ of $[n]$, there exists a permutation $\D'\in \Delta$ such that the two \LEPIs defined by $\D$ and $\D'$ are identical.
\end{Corollary}
\begin{Example}[\cref{Example2} Continued]\label{Example4}
	Consider the permutations of $[3]$ in \cref{Example2}.
	Using \cref{prop:W-equals-U}, we can verify that $(1,2,3)$, $(3,1,2)$, and $(1,3,2)$ are partial ascending permutations while $(2,1,3)$, $(2,3,1)$, and $(3,2,1)$ are not. 
	The three partial ascending permutations induce three different set collections (as shown in \cref{table1}) and three different \LEPIs (as shown in \cref{Example2}).
\end{Example}

The following theorem further shows that for any partial ascending permutation $\D \in \Delta$, the corresponding  set collection $\{W^\D_j\}_{j\in [n]}$ is unique.
\begin{Proposition}\label{lem:D1-D2}
	Given two distinct permutations $\D^1,\D^2\in \Delta$,
	their set collections satisfy $\{W^{\D^1}_j\}_{j\in [n]}\neq \{W^{\D^2}_j\}_{j\in [n]}$. 
\end{Proposition}
\begin{proof}
See Appendix \ref{appendixA}.
\end{proof}
Although different partial ascending permutations must define different set collections (as stated in \cref{lem:D1-D2}), they may still yield the same \LEPI, as demonstrated in the following example.
\begin{Example}\label{example4}
	Letting $n=4$, $N_1=\{1,2\}$, $N_2=\{3,4\}$, and $a=(1,2,3,4)^\top$, then $X_0$ reduces to 
	\begin{equation*}
		X_0=\{(w,x)\in \R\times\{0,1\}^4\,:\, w\geq f(x_1+2x_2+3x_3+4x_4),
		~x_1+x_2\leq 1, ~x_3+x_4\leq 1\}.
	\end{equation*}
	Consider the concave function $f(z) = \min\{ 2z , z+5 \}$.
	One can verify that permutations $(1,2,3,4)$ and $(1,3,2,4)$ are partial ascending, with corresponding set collections $\{\{1\},\{2\},\{2,3\},\{2,4\}\}$ and $\{\{1\},\{1,3\},\{2,3\},\{2,4\}\}$, respectively. 
	Using \cref{thm:lifting} and \cref{prop:W-equals-U}(iv), it is simple to see that the two distinct partial ascending permutations, however, define the same \eqref{LEPI}:
	\begin{equation*}
		w\geq 2x_1+4x_2+6x_3+7x_4.
	\end{equation*} 
\end{Example}

\subsection{An efficient algorithm for computing \LEPIs corresponding to partial ascending permutations}\label{section4.3}

By \cref{coro:validD}, to characterize all \LEPIs, it suffices to characterize \LEPIs defined by partial ascending permutations.
In the following, we further show that different from the general case which requires $\mathcal{O}(n^2)$ time (see \cref{prop:complex}), \eqref{LEPI} defined by any permutation $\D\in \Delta$ can be computed in $\mathcal{O}(n)$ time.
To achieve this, we need the following two lemmas.  
\begin{Lemma}\label{lem:valideta}
	Given any $\D\in \Delta$ and $j\in [n]$, it follows that $f(a(U^\D_j))=\sum_{i\in U^\D_j} \eta^\D_i$, or equivalently, 
	\begin{equation}\label{coe:valid}
		\eta^\D_{\D_j}
		=f(a(U^\D_j))-\sum_{i\in U^\D_j\backslash \D_j}\eta^\D_i,~\forall~j \in [n].
	\end{equation}
\end{Lemma}
\begin{proof}
	The result follows from \cref{thm:lifting} and \cref{prop:W-equals-U}(iv).
\end{proof}
\begin{Lemma}\label{lem::Uj-Uj-1}
	Given any $\D\in \Delta$, $j \in [n]$, and $k \in [t]$ with $\D_j \in N_k = \{i_{k - 1} + 1, i_{k - 1} + 2, \ldots, i_{k}\}$, it follows that 
	\begin{itemize}
		\item [(i)] if $\D_j = i_{k-1}+1$, then $U^\D_j = U^\D_{j - 1} \cup \{\D_j\}$;
		\item [(ii)] if $i_{k - 1} + 2 \leq \D_j \leq i_{k}$, then $\D_j -1 \in  U^\D_{j - 1}$ and $U^\D_j = U^\D_{j - 1} \backslash \{\D_j - 1\} \cup \{\D_j\}$.
	\end{itemize} 
\end{Lemma}
\begin{proof}
	The result follows from the definition of $U_j^\D$ in \eqref{def:U} and the definition of partial ascending permutations in \cref{Def:validperm}. 
\end{proof}

Given any $\D\in \Delta$ and $j\in [n]$, $k\in [t]$ with $\D_j\in N_k$, it follows from \cref{lem:valideta,lem::Uj-Uj-1} that
\begin{equation}\label{computeeta1}
	\eta^\D_{\D_j} = \left\{
	\begin{aligned}
		& f(a(U^\D_j))-f(a(U^\D_{j-1})),  && ~\text{if}~ \D_j = i_{k - 1} + 1, \\
		& f(a(U^\D_j))-f(a(U^\D_{j-1}))+\eta^\D_{\D_j-1},  && ~\text{if}~ i_{k - 1} + 2 \leq \D_j \leq i_{k}.
	\end{aligned}
	\right.
\end{equation}
Using \eqref{computeeta1}, we can design an $\mathcal{O}(n)$-time  algorithm to  compute $\{\eta^\D_{\D_j}\}_{j\in [n]}$. 
Specifically, letting $A_j=a(U^\D_j)$ for $j\in 0\cup [n]$, then, by \cref{lem::Uj-Uj-1}, $\{A_j\}_{j \in 0 \cup [n]}$ can be computed in $\CO(n)$ time using the following recursive formula
\begin{equation*}
	A_0=0, \quad A_j=  \left\{
	\begin{aligned}
		& A_{j-1} + a_{\D_j},  && ~\text{if}~ \D_j = i_{k - 1} + 1, \\
		& A_{j-1} + a_{\D_j}- a_{\D_j -1} && ~\text{if}~  i_{k - 1} + 2 \leq \D_j \leq i_{k},
	\end{aligned}
	\right.~\forall~j \in [n].
\end{equation*}
As a result, \eqref{computeeta1} reduces to
\begin{equation}\label{computeeta3}
	\eta^\D_{\D_j} = \left\{
	\begin{aligned}
		& f(A_{j})-f(A_{j-1}),  && ~\text{if}~ \D_j = i_{k - 1} + 1, \\
		& f(A_{j})-f(A_{j-1})+\eta^\D_{\D_j-1}, && ~\text{if}~  i_{k - 1} + 2 \leq \D_j \leq i_{k}.
	\end{aligned}
	\right.
\end{equation}
Note that from \cref{lem::Uj-Uj-1}, if $i_{k - 1} + 2 \leq \D_j \leq i_{k}$,  then $\D_j-1\in U^\D_{j-1}\subseteq \D(j-1)$.
Thus, by \eqref{computeeta3}, we can compute $\{\eta^\D_{\D_j}\}_{j\in [n]}$ sequentially in the order $\{\eta^\D_{\D_1}, \eta^\D_{\D_2}, \dots, \eta^\D_{\D_n}\}$, which takes $\CO(n)$ time.

\begin{Theorem}\label{thm:coecompute}
	Given any permutation $\D\in \Delta$, the coefficients $\{\eta^\D_{\D_j}\}_{j\in[n]}$ of \eqref{LEPI} defined by $\D$ can be computed in $\mathcal{O}(n)$ time.
\end{Theorem}

\section{Linear description of conv$(X_0)$ and exact separation}\label{section5}
In this section, we further present two key properties of the \LEPIs (defined by partial ascending permutations):
(i) all \LEPIs, together with bound and \GUB constraints, are able to fully characterize $\conv(X_0)$; and (ii) the separation of the exponential family of the \LEPIs (or equivalently, $\conv(X_0)$) can be implemented in $\CO(n\log n)$ time. 
To achieve this, we first investigate the separation problem of $\conv(X_0)$.

Given a point $(w^*,x^*) $, the separation problem of $\conv(X_0)$ attempts to either find a valid inequality that is violated by $(w^*, x^*)$, or to prove that no such inequality exists (i.e., $(w^*, x^*) \in \conv(X_0)$).
Since the trivial inequalities $0\leq x_i\leq 1$ for $i\in [n]$ and $\sum_{i\in N_k}x_i\leq 1$ for $k\in [t]$  can be  checked for violation in linear time, we can, without loss of generality, assume that $(w^*,x^*) \in \R \times [0, 1]^n$ with $\sum_{i \in N_k} x_i^* \leq 1$ for all $k \in [t]$.
The following lemma characterizes a property of  non-trivial facet-defining inequalities for $\conv(X_0)$.
\begin{Lemma}\label{lem:non-trivialfacet}
	Any non-trivial facet-defining  inequality $\pi_0 + \sum_{i=1}^n\pi_ix_i \leq \alpha w$ for $\conv(X_0)$ satisfies $\alpha=1$ up to scaling.
\end{Lemma}
\begin{proof}
	Since $(w,x)=(1,\boldsymbol{0})$ is a ray of $\conv(X_0)$, $\alpha$ must be non-negative.
	When $\alpha=0$, for $\pi_0 + \sum_{i=1}^n\pi_ix_i \leq 0$ to be facet-defining for $\conv(X_0)$,  it follows that
	\begin{equation*}
		\begin{aligned}
			-\pi_0=&\max_x\left\{\sum_{i=1}^n \pi_ix_i\,:\, (w,x)\in X_0\right\}\\
			=&\max_x\left\{\sum_{k=1}^t\sum_{i\in N_k} \pi_ix_i\,:\,x\in \{0,1\}^n,
			~\sum_{i\in N_k}x_i\leq 1,~\forall~k\in [t]\right\}=
			\sum_{k=1}^t \max\left\{\max_{i\in N_k}\pi_i, 0\right\}.
		\end{aligned}
	\end{equation*}
	Thus, inequality $\pi_0 + \sum\limits_{i=1}^n\pi_ix_i \leq 0$, or equivalently, $\sum_{k=1}^t\left(\sum\limits_{i\in N_k}\pi_ix_i 
		- \max\left\{\max\limits_{i\in N_k}\pi_i, 0\right\} \right)\leq 0$, 
	is dominated by trivial inequalities $\sum_{i \in N_k} x_i \leq 1$ for $ k \in [t]$ and $x_i \geq 0$ for $i\in [n]$, 
	a contradiction with the fact that $\pi_0 + \sum_{i=1}^n\pi_ix_i \leq 0$ is a non-trivial facet-defining inequality for $\conv(X_0)$.
	Therefore, $\alpha>0$ must hold, and hence inequality $\pi_0 + \sum_{i=1}^n\pi_ix_i \leq \alpha w$ can be scaled so that $\alpha=1$. 
\end{proof}

Using \cref{lem:non-trivialfacet}, we can reformulate the separation problem for $\conv(X_0)$ as an optimization problem: 
\begin{equation}
	\label{problem:separation}
	\tag{\text{SEP}}
	\max_{\pi_0,\, \pi}\left\{\pi_0+\sum_{i=1}^n\pi_ix^*_i\,:\,(\pi_0, \pi)\in \R\times \R^n,~
	\pi_0+\sum_{i=1}^n\pi_ix_i\leq w,~\forall~(w, x)\in X_0 \right\}.
\end{equation}
Letting $(\bar{\pi}_0, \bar{\pi})\in \R\times \R^n$ be an optimal solution of problem \eqref{problem:separation},
if $w^* \geq \bar{\pi}_0 + \sum_{i = 1}^{n} \bar{\pi}_i x^*$, 
then $(w^*, x^*) \in \conv(X_0)$; otherwise, $(w^*, x^*)$ violates the valid inequality $w \geq \bar{\pi}_0 + \sum_{i = 1}^{n} \bar{\pi}_i x$.
Given a point $(w, x) \in X_0 \subseteq \R \times \{0, 1\}^n$, we denote the support of $x$ by $S = \{i \in [n] \,:\, x_i = 1\}$.
Then, by the definition of $X_0$, 
it follows that $(w, x) \in X_0$ holds if and only if $
S \in \X$ and $w \geq f(a(S))$ hold, where 
\begin{equation}\label{Xdef}
	\X:= \{ S \subseteq [n]\, : \ |S \cap N_k| \leq 1,~\forall~k \in [t] \}.
\end{equation}
Therefore, problem \eqref{problem:separation} is equivalent to
\begin{equation}
	\tag{\text{SEP}'}
	\label{prob:equi-separation}
	\max_{\pi_0,\, \pi}\left\{\pi_0+\sum_{i=1}^n\pi_ix^*_i\,:\,(\pi_0, \pi)\in \R\times \R^n,~
	\pi_0+\sum_{i\in S}\pi_i\leq f(a(S)),~\forall~S\in\X \right\}.
\end{equation}
Next, we attempt to find an optimal solution of the separation problem \eqref{prob:equi-separation}.
To do this, we first define a vector $y^* \in\R^n$ as follows:
\begin{equation}\label{def:y}
	y_{\ell}^* =\sum_{i=\ell}^{i_k}x^*_i,~\forall~\ell\in N_k=\{i_{k-1}+1,\dots, i_k\},~k\in [t].
\end{equation}
Combining \eqref{def:y}, $\sum_{i \in N_k} x^*_i \leq 1$ for  $k \in [t]$, and $x^*_i\geq 0$ for $i\in [n]$, we obtain
\begin{equation}\label{property:y} 
	1\geq y_{i_{k-1}+1}^* \geq y_{i_{k-1}+2}^* \geq \dots \geq y_{i_k}^* \geq 0,~\forall~ k\in [t].
\end{equation}
Let  $\D\in \Delta$ be a partial ascending permutation satisfying
\begin{equation}\label{permY}
	y^*_{\D_1}\geq y^*_{\D_2}\geq \dots \geq y^*_{\D_n}.
\end{equation}
Note that such a permutation exists. 
Indeed, we can first find a permutation $\D'$ satisfying \eqref{permY} by sorting $\{y_i^*\}_{i \in [n]}$. 
If $\D'\in \Delta$, we are done. 
Otherwise, from \cref{Def:validperm}, there must exist $i, j \in [n]$ and some $k \in [t]$ such that $i < j $, $\D'_i, \D'_j \in N_k$, and $\D'_i > \D'_j$. 
By  $i < j$ and \eqref{permY}, it follows  $y^*_{\D'_i} \geq y^*_{\D'_j}$; and by $
\D'_i > \D'_j$, $\D'_i, \D'_j \in N_k$, and \eqref{property:y}, it follows $y^*_{\D'_i} \leq y^*_{\D'_j}$. 
As a result, $y_{\D'_i}^* = y_{\D'_j}^*$ must hold, and thus swapping the values of $y^*_{\D'_i}$ and $y^*_{\D'_j}$ still yields a permutation satisfying \eqref{permY}. 
Therefore,
\begin{Remark}\label{remark4}
	For a permutation $\bar{\D}$ satisfying \eqref{permY}, setting $\D_j := i_{k-1} +  |\bar{\D}(j-1) \cap N_k|+1$  for each $j \in [n]$ with $\bar{\D}_j\in N_k$ will yield a permutation $\D\in \Delta$ satisfying \eqref{permY}. 
	As a result, transforming a permutation $\bar{\D}$ satisfying \eqref{permY} into a permutation $\D\in \Delta$ satisfying \eqref{permY} can be conducted in $\CO(n)$ time.
\end{Remark}

We are now ready to present the main result of this section, i.e., deriving an optimal solution of the separation problem \eqref{prob:equi-separation}.
\begin{Theorem}\label{lem:solveseparation}
	Given a point $(w^*,x^*) \in \R \times [0, 1]^n$ with $\sum_{i \in N_k} x_i^* \leq 1$ for all $k \in [t]$, let $y^*$ be defined as in \eqref{def:y}, $\D\in \Delta$ be a partial ascending permutation such that $y^*_{\D_1}\geq y^*_{\D_2}\geq \dots \geq y^*_{\D_n}$, and vector $\eta^\D = (\eta^\delta_1, \ldots, \eta^\delta_n)$ be defined as in \eqref{coe:valid}. 
	Then $(\pi_0, \pi)=(0, \eta^\D)$ is an optimal solution to problem \eqref{prob:equi-separation}.
\end{Theorem}
\begin{proof}
	From \cref{lem:validity},  $\sum_{i \in S} \eta^{\D}_i \leq f(a(S))$ holds for all $S \in \X$ (defined in \eqref{Xdef}), and thus $(\pi_0, \pi)=(0, \eta^\D)$ is a feasible solution of problem \eqref{prob:equi-separation}. 
	To establish the optimality, we consider the dual of the linear programming problem \eqref{prob:equi-separation}:
	\begin{equation}
		\label{prob:dual-separation}
		\tag{\text{D-SEP'}}
		\min_{\lambda}\left\{\sum_{S\in \X}\lambda_Sf(a(S))\,:\, \sum_{S\in \X :i \in S}\lambda_S=x^*_i,~\forall~ i \in [n], ~\sum_{S\in \X}\lambda_S=1, ~\lambda_S\geq 0,~\forall~ S\in \X\right\}.
	\end{equation}
	Define the vector $\lambda^*$ as follows:
	\begin{equation}\label{lambda}
		\lambda^*_S=
		\begin{cases}
			y^*_{\D_j}-y^*_{\D_{j+1}}, & ~\text{if}~ S=U^\D_j\text{~holds  for some } j\in 0 \cup [n],\\[5pt]
			0, & ~\text{otherwise},
		\end{cases}
		\quad \forall~S\in \X,
	\end{equation}
	where $U_{j}^\D$ is defined in \eqref{def:U},  $y^*_{\D_0}:=1$, and $y^*_{\D_{n+1}}:=0$.
	By \eqref{permY}, it follows that  $y^*_{\D_0} \geq y^*_{\D_1} \geq \cdots \geq  y^*_{\D_{n + 1}}$ and thus $\lambda^*_S\geq 0$ for all $S\in \X$.
	Moreover, $\sum_{S\in \X}\lambda^*_S=\sum_{j=0}^n(y^*_{\D_j}-y^*_{\D_{j+1}})
	=y^*_{\D_0}-y^*_{\D_{n+1}}=1-0=1$.
	Then, to show that $\lambda^*$ is a feasible solution to \eqref{prob:dual-separation}, it remains to prove that $\sum_{S\in \X :i \in S}\lambda^*_S=x^*_i$ holds for all $i \in [n]$,
	or equivalently, 
	\begin{equation}\label{eq:dual-cons}
		\sum_{j\,:\, i\in U^\D_j} (y^*_{\D_j} - y^*_{\D_{j + 1}})
		=x^*_i,~\forall~ i \in [n].
	\end{equation}
	Given any $r\in [n]$, let $k \in [t]$ be such that $r\in N_k=\{i_{k-1}+1,\dots,i_k\}$.
	We prove that \eqref{eq:dual-cons} holds at $i=r$ by considering the two cases $i_{k-1} + 1\leq r < i_k$ and $r=i_k$, separately.\\[8pt]
	(i) $i_{k-1} + 1 \leq r<i_k$.
	Then $i_{k-1} + 1 \leq r<r+1\leq i_k$, and thus $r,r+1\in N_k$.
	Let $j_1,j_2$ be such that $r=\D_{j_1}$ and $r+1=\D_{j_2}$.
	From $\D\in \Delta$, we have (a) $j_1<j_2$; (b)  $\max_{i\in N_k\cap \D(j)} i < \D_{j_1}=r$ for $j < j_1$; (c) $\max_{i\in N_k\cap \D(j)} i = \D_{j_1}=r$ for $j_1 \leq j < j_2$; and (d) $\max_{i\in N_k\cap \D(j)} i \geq \D_{j_2}> \D_{j_1}=r$ for $j \geq j_2$.
	These, together with the definition of $U_j^\D$ in \eqref{def:U}, imply that $\{j\in [n]\,:\, r\in U^\D_j\}=\{j_1,\dots,j_2-1\}$.
	Then we obtain
	\begin{equation*}
		\sum_{j\,:\, r\in U^\D_j}(y^*_{\D_j}-y^*_{\D_{j+1}})=
		\sum_{j=j_1}^{j_2-1}(y^*_{\D_j}-y^*_{\D_{j+1}})
		=y^*_{\D_{j_1}}-y^*_{\D_{j_2}}=y^*_{r}-y^*_{r+1}
		\stackrel{(a)}{=}\sum_{i=r}^{i_k}x^*_i-\sum_{i=r+1}^{i_k}x^*_i=x^*_{r},
	\end{equation*}
	where (a) follows from \eqref{def:y}.
	\\[8pt]
	(ii) $r=i_k$. Let $j_1\in [n]$ be such that $i_k=\D_{j_1}$.
	Similar to case (i), we can derive $\{j\in [n] \,:\, i_k\in U^\D_j\}=\{j_1,\dots, n\}$.
	Therefore,
	\begin{equation*}
		\begin{aligned}
			\sum_{j \,:\, i_k\in U^\D_j}(y^*_{\D_j}-y^*_{\D_{j+1}})= \sum_{j=j_1}^n(y^*_{\D_j}-y^*_{\D_{j+1}})
			= y^*_{\D_{j_1}}-y^*_{\D_{n+1}}=y^*_{i_k}-0= 
			\sum_{i=i_k}^{i_k}x^*_i=x^*_{i_k}.
		\end{aligned}
	\end{equation*}
	Combining the above two cases, \eqref{eq:dual-cons} must hold, and hence $\lambda^*$ is a feasible solution to the dual problem \eqref{prob:dual-separation}.
	
	To prove the optimality of $(0, \eta^\D)$, by strong duality of an \LP problem, it suffices to prove
	\begin{equation}\label{eq:obj-equal}
		\sum_{S\in \X}\lambda^*_Sf(a(S))=\sum_{i\in [n]}\eta^\D_i x^*_i.
	\end{equation}
	Observe that
	\begin{small}
	\begin{equation}\label{comp1}
		\begin{aligned}
		&\sum_{S\in \X}\lambda^*_Sf(a(S))= \sum_{j=0}^n
		 (y^*_{\D_j}-y^*_{\D_{j+1}}) f(a(U^\D_j))
		 \stackrel{(a)}{=}
		  \sum_{j\in [n]}[f(a(U^\D_j))-f(a(U^\D_{j-1}))]y^*_{\D_j}\\
		 &\qquad \stackrel{(b)}{=}
		 \sum_{j:\D_j\in [n]}[f(a(U^\D_j))-f(a(U^\D_{j-1}))]y^*_{\D_j}
		 =\sum_{k=1}^t \sum_{j:i_{k-1} + 1\leq\D_j\leq i_k}
		 	[f(a(U^\D_j))-f(a(U^\D_{j-1}))]\sum_{i=\D_j}^{i_k} x^*_i\\
		 &\qquad =\sum_{k=1}^t \sum_{i=i_{k-1}+1}^{i_k} x^*_i
		 \sum_{j:i_{k-1}+1\leq\D_j\leq i} \left[f(a(U^\D_j))-f(a(U^\D_{j-1}))\right],
		 \end{aligned}
	\end{equation}
	\end{small}%
	where (a) follows from $f(a(U^\D_0))=f(0)=0$ and $y^*_{\D_{n+1}}=0$; (b) follows from $\{\D_1, \ldots, \D_n\} = [n]$.
    Note that, for any $i\in [n]$ and $k\in [t]$ satisfying $i\in N_k$,
	\begin{equation}\label{comp2}
			\begin{aligned}
				&\sum_{j:i_{k-1}+1\leq\D_j\leq i}\left[f(a(U^\D_j))-f(a(U^\D_{j-1}))\right]
				\stackrel{(a)}{=} \sum_{j:i_{k-1}+1\leq\D_j\leq i}
				\left(\sum_{\ell\in U^\D_j}\eta^\D_{\ell}-
				\sum_{\ell\in U^\D_{j-1}} \eta^\D_{\ell} \right)\\
				&\qquad \stackrel{(b)}{=} \eta^\D_{i_{k-1}+1} + \sum_{j:i_{k-1}+2\leq\D_j\leq i}
				(\eta^\D_{\D_j}-\eta^\D_{\D_j-1})
				=\eta^\D_{i},
			\end{aligned}
	\end{equation}
	where (a) follows from $\D\in \Delta$ and \cref{lem:valideta}; (b) follows from \cref{lem::Uj-Uj-1}.
	Combining \eqref{comp1} and \eqref{comp2}, we obtain
	\begin{equation*}
		\sum_{S\in \X} \lambda^*_S f(a(S))= \sum_{k=1}^t
		\sum_{i=i_{k-1}+1}^{i_k} x^*_i \eta^\D_{i}
		= \sum_{i=1}^n\eta^\D_ix^*_i.
	\end{equation*}
	As a result, \eqref{eq:obj-equal} holds and the statement follows.
\end{proof}

Three remarks on \cref{lem:solveseparation} are in order.
First, \cref{lem:solveseparation} implies that given a point $(w^*,x^*) \in \R \times [0, 1]^n$ with $\sum_{i \in N_k} x_i^* \leq 1$ for all $k \in [t]$, either the \LEPI $w \geq \sum_{i=1}^n \eta_{\D_i}^\D x_{\D_i} $ (where $(0, \eta^\D)$ is optimal to \eqref{prob:equi-separation}) cuts off $(w^*,x^*)$ or none exists. 
This indicates the strength of \LEPIs in terms of characterizing the convex hull of $X_0$.
\begin{Corollary}\label{coro:charaX0}
	All \LEPIs $w\geq \sum_{i=1}^n\eta^{\D}_{\D_i}x_{\D_i}$, $\D \in \Delta$, along with the trivial inequalities $\sum_{i\in N_k}x_i\leq 1$, $k\in [t]$, and $0\leq x_i\leq 1$, $i\in [n]$, provide a complete linear description of  $\conv(X_0)$.
\end{Corollary}
Second, \cref{lem:solveseparation} enables us to derive an efficient polynomial-time algorithm for the separation of the \LEPIs (or equivalently, $\conv(X_0)$).
Indeed, 
given a point $(w^*,x^*) \in \R \times [0, 1]^n$ with $\sum_{i \in N_k} x_i^* \leq 1$ for all $k \in [t]$, we can use the following exact procedure to solve the separation problem.
\begin{itemize}
	\item [(i)] Compute $y^*$ in \eqref{def:y} with the time complexity of $\CO(n)$.
	\item [(ii)] Determine the permutation $\D\in \Delta$ satisfying \eqref{permY} by first sorting $\{y_i^*\}_{i \in [n]}$ with the time complexity of $\CO(n \log n)$ and then transforming the obtained permutation $\D'$ into a permutation $\D\in \Delta$ satisfying \eqref{permY} with the time complexity  of  $\CO(n)$; see \cref{remark4}.
	\item [(iii)] Compute the coefficients $\{\eta_{\D_j}^\D\}_{j\in [n]}$ of \eqref{LEPI} corresponding to $\D$ ($\in \Delta$) with the time complexity of  $\CO(n)$; see \cref{thm:coecompute}.
	\item [(iv)] Check whether or not \eqref{LEPI} is violated by $(w^*, x^*)$ with the time complexity of $\CO(n)$.
\end{itemize}
\begin{Corollary}\label{coro:separationX0}
	The separation problem of the \LEPIs~{\rm(}or $\conv(X_0)${\rm)} can be solved in $\CO(n\log n)$ time.
\end{Corollary}

\begin{Remark}
	When $|N_k|=1$ holds for all $k\in [t]$, we have $y^*_i=x^*_i$ for all $i\in [n]$ in step (i).
	Then, step (ii) sorts the values $\{x_i^*\}_{i\in [n]}$, and step (iii) computes the coefficients of \eqref{LEPI} that is equivalent to \eqref{EPI}.
	In this case,  our exact separation algorithm reduces to the separation of the \EPIs and is indeed equivalent to the well-known greedy algorithm proposed by \citet*{Edmonds2003}. 
\end{Remark}

Third, due to the polynomial equivalence of optimization and separation \citep{Grotschel2012}, \cref{coro:separationX0} implies that optimizing a linear function over $X_0$ can be conducted in polynomial time.
\begin{Corollary}\label{coro:optimize}
	The optimization problem 
	\begin{equation}\label{opt}\tag{OPT}
	\min_{w, \, x}\left\{d w+c^\top x\,:\, ~w \geq f(a^\top x ),~ \sum_{i\in N_k}x_i \leq 1,~\forall~ k\in [t],~(w,x)\in \R \times \{0,1\}^n\right\}
	\end{equation}
	 can be solved in polynomial time, where $d\in \R$, $c\in \R^n$, $a\in \R^n_+$, and $\{N_k\}_{k\in [t]}$ is a partition of $[n]$.
\end{Corollary}
\noindent In Appendix \ref{appendixB}, we further provide a strongly polynomial-time algorithm for solving problem \eqref{opt}.

\section{Extensions to the general case conv$(X)$}\label{section6}

In this section, we extend the previous polyhedral results of $\conv(X_0)$ onto the general case $\conv(X)$, where we recall that
\begin{equation*}
	X=\left\{(w,x)\in \R \times\{0,1\}^n\,:\, w \geq f(a^\top x)+b^\top x,
	~\sum_{i \in N_k}x_i \leq 1, ~\forall~ k\in[t]\right\}.
\end{equation*}
To proceed, we note that there exists a one-to-one correspondence between the feasible points of $X_0$ and $X$.
\begin{Lemma}\label{lem:point}
	Let $(w,x)\in\R\times \R^n$.
	Then $(w,x)\in X$ if and only if $(w-b^\top x, x)\in X_0$.
\end{Lemma}
\begin{proof}
	This statement follows from the definition of $X$ and $X_0$. 
\end{proof}
\begin{Lemma}\label{lem:independent}
	Let $\{(w^j,x^j)\}_{j\in [n+1]}\subseteq \R\times \R^n$.
	Then $\{(w^j,x^j)\}_{j\in [n+1]}$ are affinely independent if and only if 
	$\{(w^j-b^\top x^j, x^j)\}_{j\in [n+1]}$ are affinely independent.
\end{Lemma}
\begin{proof}
	It suffices to show that the following statements are equivalent:
	\begin{itemize}
		\item [(i)] $\{(w^j,x^j)\}_{j\in [n+1]}$ are affinely independent.
		\item [(ii)] $\alpha=\boldsymbol{0}$ is the unique feasible solution to the following linear system:
		\begin{equation}\label{tmpeqx1}
		\sum_{j=1}^{n+1} \alpha_j=0, ~\sum_{j=1}^{n+1} \alpha_jx^j_i=0, ~\forall~ i\in[n], ~\text{and}~ \sum_{j=1}^{n+1} \alpha_jw^j =0.
	    \end{equation}
		\item [(iii)] $\alpha=\boldsymbol{0}$ is the unique feasible solution to the following linear system:
		\begin{equation}\label{tmpeqx2}
			\sum_{j=1}^{n+1} \alpha_j=0, ~\sum_{j=1}^{n+1} \alpha_jx^j_i=0, ~\forall~ i\in[n], ~\text{and}~ \sum_{j=1}^{n+1} \alpha_j(w^j-b^\top x^j) =0.
		\end{equation}
		\item [(iv)] $\{(w^j-b^\top x^j, x^j)\}_{j\in [n+1]}$ are affinely independent.
	\end{itemize}
	The equivalence of (i) and (ii) (respectively, (iii) and (iv)) is trivially satisfied.
	For (ii) $\Leftrightarrow$ (iii), observe that for solutions satisfying the first two constraints in the linear systems \eqref{tmpeqx1} and \eqref{tmpeqx2}, it follows 
	$\sum_{j=1}^{n+1} \alpha_j(w^j- b^\top x^j)= \sum_{j=1}^{n+1} \alpha_j w^j-
    \sum_{i=1}^n b_i \sum_{j=1}^{n+1} \alpha_j x^j_i= \sum_{j=1}^{n+1}\alpha_j w^j$,
    and thus the third constraints in \eqref{tmpeqx1} and \eqref{tmpeqx2} are equivalent. 
\end{proof}

The following theorem further establishes a one-to-one correspondence between the facet-defining inequalities of $\conv(X_0)$ and $\conv(X)$.
\begin{Theorem}\label{thm:correspondfacet}
	Inequality 
	\begin{equation}\label{equivalent-1}
	\pi_0+\pi^\top x \leq \alpha w
    \end{equation}
    is valid (respectively, facet-defining) for $\conv(X_0)$ if and only if
    \begin{equation}\label{equivalent-2}
    	\pi_0+\pi^\top x\leq \alpha(w-b^\top x)
    \end{equation}  
	is valid (respectively, facet-defining) for $\conv(X)$.
\end{Theorem}
\begin{proof}
    We shall prove the ``if'' part and the proof of the ``only if'' part is similar.
    Assume that inequality \eqref{equivalent-2} is valid for $\conv(X)$.
    Let $(w^*,x^*)$ be an arbitrary point in $X_0$.
    Then, by \cref{lem:point}, it follows $(w^*+b^\top x^*, x^*)\in X$.
    Substituting $(w^*+b^\top x^*, x^*)$ into \eqref{equivalent-2}, we obtain  $\pi_0+\pi^\top x^*\leq \alpha w^*$.
    Since $(w^*,x^*)$ is chosen arbitrarily from $X_0$, it follows that \eqref{equivalent-1} is a valid inequality for $\conv(X_0)$.
    If \eqref{equivalent-2} defines a facet of $\conv(X)$,
    then there exist $n+1$ affinely independent points $\{(w^j,x^j)\}_{j\in [n+1]}\in X$ such that 
    $\pi_0+\pi^\top x^j= \alpha (w^j- b^\top x^j)$ holds for all $j \in [n+1]$.
    From \cref{lem:point,lem:independent}, $\{(w^j-b^\top x^j,x^j)\}_{j \in [n+1]} \in X_0$ are $n+1$ affinely independent points satisfying \eqref{equivalent-1} at equality,
    showing that inequality \eqref{equivalent-1} is facet-defining for $\conv(X_0)$. 
\end{proof}

Using  \cref{thm:correspondfacet}, we can extend the \EPIs and  \LEPIs of $\conv(X_0)$ onto the general case $\conv(X)$. 
Specifically, for any permutation $\D$ of $[n]$, by \cref{thm:correspondfacet} and the validity of \eqref{EPI} and \eqref{LEPI}, we obtain the \EPI-like and \LEPI-like (valid) inequalities for $\conv(X)$:  
\begin{align}
	&w\geq \sum_{j=1}^n(f(a(\D(j)))-f(a(\D(j-1)))+b_{\D_j})x_{\D_j},\label{EPI-like}\tag{EPI'}\\
	& w\geq \sum_{j=1}^n(\eta^\D_{\D_j} +b_{\D_j})x_{\D_j}.\label{LEPI-like}\tag{LEPI'}
\end{align}
Moreover, combining \cref{lem:facet}, \cref{thm:LEPI-EPI}, \cref{coro:charaX0}, and \cref{thm:correspondfacet}, we obtain that
\begin{Corollary}\label{tmpcoro}
	(i) For any permutation $\D$ of $[n]$,
 	inequality \eqref{LEPI-like} is facet-defining for $\conv(X)$ and stronger than inequality \eqref{EPI-like};
    (ii) All \LEPI's together with the trivial inequalities $0\leq x_i\leq 1$ for $i\in [n]$ and $\sum_{i\in N_k}x_i\leq 1$ for $k\in [t]$ provide a complete linear description of $\conv(X)$.
\end{Corollary}

Let $(w^*,x^*) \in \R \times [0, 1]^n$ be such that $\sum_{i \in N_k} x_i^* \leq 1$ holds for all $k \in [t]$. 
Using \cref{thm:correspondfacet}, we can derive that  $(w^*,x^*)\in \conv(X)$ holds if and only if $(w^*-b^\top x^*, x^*)\in \conv(X_0)$ holds. 
As a result, separating $(w^*,x^*)$ from $\conv(X)$ (respectively, determining whether there exists an inequality \eqref{LEPI-like} cutting off $(w^*,x^*)$) is equivalent to separating  $(w^*-b^\top x^*, x^*)$ from $\conv(X_0)$  (respectively, determining whether there exists an inequality \eqref{LEPI} cutting off $(w^*-b^\top x^*, x^*)$).
Therefore, it follows from \cref{coro:separationX0} that
\begin{Corollary}
	The separation problem of the \LEPI's (or $\conv(X)$) can be solved in $\CO(n\log n)$ time.
\end{Corollary}

\section{Computational Results}\label{section7}
In this section, we present computational results to demonstrate the effectiveness of the proposed \LEPIs or \LEPI's in a branch-and-cut framework 
for solving problems \eqref{cGC-MCLP} and \eqref{problemSOCP}.
All computational experiments were performed on a cluster of Intel(R) Xeon(R) Gold 6230R CPU @ 2.10 GHz computers using CPLEX 20.1.0 as the solver.
We implemented \LEPIs or \LEPI's as cutting planes using the callback function of CPLEX.
CPLEX was set to run in a single-threaded mode, with a time limit of 7200 seconds and a relative gap tolerance of 0\%.
Unless otherwise specified, all other CPLEX parameters were set to their default values.

\subsection{Minimum probabilistic covering location problem}\label{section7.1}
We first consider problem \eqref{cGC-MCLP}, where we recall that the substructures $\{\bar{X}^i_0\}_{i \in I}$ in \eqref{substructurebarX0i} can be represented by the classic \EPIs of \citet{Edmonds2003} or the newly proposed \LEPIs (with $\sum_{s\in S}x_{js}\leq 1$ for $j\in J$).
We use a similar procedure as in \citet{Karatas2021} to construct the \cGCMCLP instances. 
Specifically,
\begin{itemize}
	\item The locations of customers and facilities are randomly generated on a two-dimensional 100×100 grid from a multivariate uniform distribution.
	\item The weight $v_i$ is an integer uniformly chosen from $[1,100]$.
	\item The probability $p_{ijs}$ is calculated using a Fermi-type coverage function:
	\begin{equation*}
		p_{ijs}=
		\begin{cases}
			1, & \text{if}~d_{ij}\leq {d^{\text{min}}_s},\\
			\frac{1}{1+10^{[2(d_{ij}-{d^{\text{min}}_s})/({{d^{\text{max}}_s}}-{d^{\text{min}}_s})-1]/\alpha}}, & \text{if}~{d^{\text{min}}_s}<d_{ij}\leq d^{\text{max}}_s,\\
			0, & \text{if}~ d_{ij}>d^{\text{max}}_s,
		\end{cases}
	\end{equation*}
	where $d_{ij}$ is the distance between the locations of customer node $i$ and facility $j$,
	${d^{\text{min}}_s}$ and ${{d^{\text{max}}_s}}$ represent the minimum and 
	the maximum coverage radii for a facility of type $s$, respectively,
	and $\alpha$ is a sensitivity parameter (set to 0.5) describing the tailing character of the coverage function.
	\item  The number of types $|S|$ of facilities is taken from $\{3,4,5,6\}$ and the corresponding service capacity requirement is taken from $\{100, 200, 300, 400\}$.
	For each $|S|\in \{3,4,5,6\}$, the related facility parameters $(c_s, d_s^\text{min}, d_{s}^\text{max})$ are taken from the first $|S|$ elements of 
	$\{ (10, 5, 10),  (20, 6, 14), (30, 7, 18), (40, 8, 22), (50, 9, 26),(60, 10, 30)\}$.
	\item The number of customer nodes and candidate facility locations, represented as $(|I|,|J|)$, is taken from $\{(100, 20), (200, 40), (300, 60), (400, 80), (500, 100)\}$.
\end{itemize}
For each combination of $|S|$ and pair $(|I|,|J|)$, we generate 10 random  instances, yielding a testbed of 200 instances.
\begin{table}
	\caption{Performance comparison of the state-of-the-art approach of \citet{Karatas2021}, the branch-and-cut algorithms based on the classic \EPIs of \citet{Edmonds2003} and the proposed \LEPIs on \cGCMCLP instances.}
	{\resizebox{\textwidth}{!}
				{\begin{tabular}{lllrrrrrrrrrrrrrrr}
					\toprule
					\multicolumn{1}{l}{\multirow{2}{*}{$|S|$}} 
					& \multicolumn{1}{l}{\multirow{2}{*}{$|I|$}} 
					& \multicolumn{1}{l}{\multirow{2}{*}{$|J|$}} 
					& \multicolumn{5}{c}{\EF} & \multicolumn{5}{c}{\EPIF} & \multicolumn{5}{c}{\LEPIF} \\
					\cmidrule(r){4-8}\cmidrule(r){9-13}\cmidrule(r){14-18}
					&&& \multicolumn{1}{c}{Solved} & \multicolumn{1}{c}{Time} & \multicolumn{1}{c}{Nodes} & \multicolumn{1}{c}{EGap(\%)} & \multicolumn{1}{c}{RGap(\%)} & \multicolumn{1}{c}{Solved} & \multicolumn{1}{c}{Time} & \multicolumn{1}{c}{Nodes} & \multicolumn{1}{c}{EGap(\%)} & \multicolumn{1}{c}{RGap(\%)} & \multicolumn{1}{c}{Solved} & \multicolumn{1}{c}{Time} & \multicolumn{1}{c}{Nodes} & \multicolumn{1}{c}{EGap(\%)} & \multicolumn{1}{c}{RGap(\%)} \\ \midrule
					3 & 100 & 20& 10 & 1 & 12 & 0.00 & 2.87 & 10 & 1 & 628 & 0.00 & 38.09 & 10 & 1 & 59 & 0.00 & 10.90 \\
					& 200 & 40& 10 & 23 & 20 & 0.00 & 3.62 & 10 & 3 & 1316 & 0.00 & 38.04 & 10 & 1 & 49 & 0.00 & 7.78 \\
					& 300 & 60& 10 & 226 & 26 & 0.00 & 4.42 & 10 & 11 & 2914 & 0.00 & 39.73 & 10 & 2 & 109 & 0.00 & 9.04 \\
					& 400 & 80& 7 & 3033 & 216 & 0.97 & 17.14 & 10 & 48 & 6822 & 0.00 & 48.05 & 10 & 7 & 698 & 0.00 & 18.95 \\
					& 500 & 100& 1 & 5971 & 28 & 20.71 & 18.53 & 10 & 161 & 12181 & 0.00 & 50.64 & 10 & 19 & 1067 & 0.00 & 23.22 \\
					4 & 100 & 20& 10 & 2 & 8 & 0.00 & 2.66 & 10 & 2 & 881 & 0.00 & 32.89 & 10 & 1 & 15 & 0.00 & 3.68 \\
					& 200 & 40& 10 & 80 & 59 & 0.00 & 9.92 & 10 & 14 & 4107 & 0.00 & 43.10 & 10 & 1 & 63 & 0.00 & 10.05 \\
					& 300 & 60& 10 & 960 & 114 & 0.00 & 13.41 & 10 & 54 & 9412 & 0.00 & 48.43 & 10 & 3 & 229 & 0.00 & 12.41 \\
					& 400 & 80& 4 & 4844 & 108 & 7.27 & 17.08 & 10 & 258 & 18649 & 0.00 & 53.44 & 10 & 15 & 487 & 0.00 & 16.15 \\
					& 500 & 100& 0 & 7200 & 21 & 32.84 & 29.96 & 10 & 699 & 34121 & 0.00 & 60.97 & 10 & 54 & 2108 & 0.00 & 29.32 \\
					5 & 100 & 20& 10 & 4 & 10 & 0.00 & 2.62 & 10 & 5 & 2715 & 0.00 & 43.13 & 10 & 1 & 12 & 0.00 & 2.09 \\
					& 200 & 40& 10 & 252 & 126 & 0.00 & 10.66 & 10 & 121 & 21554 & 0.00 & 54.97 & 10 & 3 & 123 & 0.00 & 7.67 \\
					& 300 & 60& 9 & 2135 & 148 & 0.37 & 11.41 & 10 & 362 & 26671 & 0.00 & 57.89 & 10 & 7 & 237 & 0.00 & 10.70 \\
					& 400 & 80& 2 & 6613 & 79 & 15.75 & 22.01 & 7 & 2625 & 115420 & 1.25 & 71.26 & 10 & 55 & 1752 & 0.00 & 20.29 \\
					& 500 & 100& 0 & 7200 & 0 & 69.97 & - & 4 & 5921 & 144971 & 5.50 & 78.17 & 10 & 442 & 16665 & 0.00 & 37.30 \\
					6 & 100 & 20& 10 & 17 & 132 & 0.00 & 4.27 & 10 & 25 & 15657 & 0.00 & 56.56 & 10 & 1 & 288 & 0.00 & 5.51 \\
					& 200 & 40& 10 & 1617 & 1122 & 0.00 & 16.48 & 8 & 2136 & 282846 & 0.52 & 66.43 & 10 & 9 & 1753 & 0.00 & 15.61 \\
					& 300 & 60& 1 & 6780 & 205 & 14.88 & 19.71 & 4 & 4009 & 168670 & 4.12 & 75.64 & 10 & 38 & 3473 & 0.00 & 19.90 \\
					& 400 & 80& 0 & 7200 & 2 & 49.43 & 28.26 & 1 & 6609 & 144185 & 27.29 & 81.51 & 10 & 452 & 16845 & 0.00 & 34.14 \\
					& 500 & 100& 0 & 7200 & 0 & 515.53 & - & 0 & 7200 & 90213 & 58.00 & 84.20 & 9 & 2590 & 67981 & 0.29 & 45.14 \\
					\multicolumn{3}{l}{\textbf{Average}}&& \textbf{3068} & \textbf{122} & \textbf{36.39} & \textbf{11.75} && \textbf{1513} & \textbf{55197} & \textbf{4.83} & \textbf{56.16} && \textbf{185} & \textbf{5701} & \textbf{0.01} & \textbf{16.99} \\
					\multicolumn{3}{l}{\textbf{Allsolved}} & \textbf{124} &&&&& \textbf{164} &&&&& \textbf{199} &&&&\\
					\bottomrule
			\end{tabular}}}
	\label{tablecgcmclp}
\end{table}

To illustrate the advantages of the proposed \LEPIs over the classic \EPIs of \citet{Edmonds2003} (i.e., inequalities \eqref{LEPI} over \eqref{EPI}), we compare the performance of the following two settings:
\begin{itemize}
	\item \EPIF: solving \eqref{cGC-MCLP} by the branch-and-cut algorithm based on the classic \EPIs of \citet{Edmonds2003};
	\item \LEPIF: solving \eqref{cGC-MCLP} by the branch-and-cut algorithm based on the proposed \LEPIs.
\end{itemize}
For benchmarking purposes, we also report the results of solving the state-of-the-art  \MILP formulation of \citet{Karatas2021} for \eqref{cGC-MCLP},  denoted as setting \EF.
Note that this formulation involves $\CO(|I||J||S|)$ variables and constraints.

\cref{tablecgcmclp} summarizes the performance results.
For each setting, we report the number of instances solved to optimality (Solved), the average CPU time in seconds (Time), the average number of explored nodes (Nodes), and the average end gap (EGap).
The end gap is computed as $100\% \times (\text{UB}-\text{LB})/\text{UB}$ where LB and UB are the lower and upper bounds returned by CPLEX (when the time limit was hit).
For settings \EPIF and \LEPIF, we also report the average LP relaxation gap at the root node (RGap).
The root gap is defined by $100\%\times (z_{\text{opt}}-z_{\text{root}})/z_{\text{opt}}$, where $z_{\text{opt}}$ is the objective value of optimal solution or best incumbent and $z_{\text{root}}$ is the objective value at the root node.

From \cref{tablecgcmclp}, we observe that for instances with a small problem size (i.e., small value of $|S|$, $|I|$, and $|J|$), solving the state-of-the-art \MILP formulation of  \citet{Karatas2021}  by CPLEX can already find an optimal solution for \eqref{cGC-MCLP}.
However, for instances with a large problem size, due to the huge numbers of variables and constraints, solving problem  \eqref{cGC-MCLP} using this approach usually fails to find an optimal solution.
In sharp contrast, the branch-and-cut algorithms based on classic \EPIs and the newly proposed \LEPIs work on a ``slim'' formulation with $\mathcal{O}(|I|+|J||S|)$  variables, which is one order of magnitude smaller, and add \EPIs and \LEPIs in a dynamic fashion, thereby achieving an overall better performance.
Overall,  \EPIF and \LEPIF can solve $40$ and $75$ more \cGCMCLP instances to optimality than \EF, respectively, with significantly less computational time.

Next, we compare the performance of the branch-and-cut algorithms based on the classic \EPIs and our newly proposed \LEPIs.
We note that the average LP relaxation gap at the root node under setting \EPIF is 56.16\% while that under setting \LEPIF
is only 16.99\%, which shows that compared with \EPIs, the proposed \LEPIs are much more effective in terms of providing a stronger \LP relaxation bound.
The improvement can be attributed to two favorable theoretical properties of \LEPIs: (i) they are facet-defining for $\conv(X_0)$ and stronger than \EPIs (as shown in \cref{lem:facet} and \cref{thm:LEPI-EPI}); and (ii) unlike \EPIs, which only describe $X_0$, \LEPIs provide a linear description of the convex hull of $X_0$.
This tighter LP relaxation directly translates into superior computational performance of  \LEPIF compared to \EPIF. 
In particular, equipped with the proposed  \LEPIs, 
\LEPIF can solve 35 more instances to optimality, and 
the average CPU time and number of nodes are reduced by factors of   8.2 (1513/185) and  9.7 (55197/5701), respectively.

\subsection{The multiple probabilistic knapsack problem with \GUB constraints}\label{section7.2}
\begin{table}
	\caption{Performance comparison of the branch-and-cut algorithms based on the default setting of CPLEX, inequalities \eqref{EPI-like} and \eqref{LEPI-like} on \MPKPG instances.}
	{\resizebox{\textwidth}{!}
	{\begin{tabular}{lllrrrrrrrrrrrrrrr}
		\toprule
		\multicolumn{1}{l}{\multirow{2}{*}{$\beta$}} 
		& \multicolumn{1}{l}{\multirow{2}{*}{$n$}} 
		& \multicolumn{1}{l}{\multirow{2}{*}{$m$}} 
		& \multicolumn{5}{c}{\CPX} & \multicolumn{5}{c}{\CPXEPI} & \multicolumn{5}{c}{\CPXLEPI}\\
		\cmidrule(r){4-8}\cmidrule(r){9-13}\cmidrule(r){14-18}
		&&& \multicolumn{1}{c}{Solved} & \multicolumn{1}{c}{Time} & \multicolumn{1}{c}{Nodes} & \multicolumn{1}{c}{EGap(\%)} & \multicolumn{1}{c}{RGap(\%)} & \multicolumn{1}{c}{Solved} & \multicolumn{1}{c}{Time} & \multicolumn{1}{c}{Nodes} & \multicolumn{1}{c}{EGap(\%)} & \multicolumn{1}{c}{RGap(\%)} & \multicolumn{1}{c}{Solved} & \multicolumn{1}{c}{Time} & \multicolumn{1}{c}{Nodes} & \multicolumn{1}{c}{EGap(\%)} & \multicolumn{1}{c}{RGap(\%)} \\ \midrule
		0.3& 80& 20& 10 & 49 & 14424 & 0.00 & 43.85 & 10 & 21 & 9149 & 0.00 & 36.86 & 10 & 16 & 7469 & 0.00 & 28.92 \\
		&& 30& 10 & 35 & 18181 & 0.00 & 59.13 & 10 & 25 & 10674 & 0.00 & 51.73 & 10 & 21 & 8734 & 0.00 & 40.67 \\
		&& 40& 10 & 38 & 15589 & 0.00 & 71.02 & 10 & 32 & 9990 & 0.00 & 64.25 & 10 & 27 & 9053 & 0.00 & 49.91 \\
		& 120& 20& 10 & 170 & 72743 & 0.00 & 40.60 & 10 & 62 & 29367 & 0.00 & 33.55 & 10 & 38 & 17818 & 0.00 & 24.24 \\
		&& 30& 10 & 246 & 101932 & 0.00 & 54.28 & 10 & 121 & 47614 & 0.00 & 46.68 & 10 & 73 & 29054 & 0.00 & 33.84 \\
		&& 40& 10 & 237 & 75023 & 0.00 & 63.64 & 10 & 146 & 41114 & 0.00 & 55.43 & 10 & 97 & 28548 & 0.00 & 40.21 \\
		& 160& 20& 10 & 1338 & 528805 & 0.00 & 39.06 & 10 & 284 & 109281 & 0.00 & 31.47 & 10 & 134 & 58988 & 0.00 & 21.82 \\
		&& 30& 10 & 2008 & 552463 & 0.00 & 51.47 & 10 & 623 & 165292 & 0.00 & 43.83 & 10 & 347 & 109483 & 0.00 & 30.50 \\
		&& 40& 10 & 2240 & 529906 & 0.00 & 62.83 & 10 & 926 & 212765 & 0.00 & 54.04 & 10 & 554 & 143879 & 0.00 & 38.46 \\
		0.5& 80& 20& 10 & 91 & 53603 & 0.00 & 15.03 & 10 & 17 & 5345 & 0.00 & 10.26 & 10 & 11 & 2744 & 0.00 & 7.58 \\
		&& 30& 10 & 251 & 110355 & 0.00 & 19.42 & 10 & 42 & 13539 & 0.00 & 15.20 & 10 & 22 & 6649 & 0.00 & 10.46 \\
		&& 40& 10 & 576 & 172976 & 0.00 & 22.05 & 10 & 76 & 20098 & 0.00 & 17.74 & 10 & 38 & 9610 & 0.00 & 11.56 \\
		& 120& 20& 10 & 798 & 312215 & 0.00 & 14.56 & 10 & 144 & 62141 & 0.00 & 11.55 & 10 & 41 & 16780 & 0.00 & 8.27 \\
		&& 30& 7 & 3321 & 745795 & 8.36 & 18.84 & 10 & 751 & 187878 & 0.00 & 15.77 & 10 & 159 & 47949 & 0.00 & 10.27 \\
		&& 40& 4 & 5762 & 979783 & 3.29 & 20.89 & 10 & 1248 & 233827 & 0.00 & 17.59 & 10 & 249 & 62641 & 0.00 & 11.66 \\
		& 160& 20& 6 & 2417 & 648606 & 4.55 & 9.63 & 9 & 972 & 281801 & 0.75 & 7.73 & 10 & 108 & 38050 & 0.00 & 5.99 \\
		&& 30& 0 & 7200 & 1187223 & 4.81 & 16.68 & 4 & 5187 & 872784 & 3.20 & 14.62 & 10 & 727 & 180941 & 0.00 & 10.81 \\
		&& 40& 0 & 7200 & 812611 & 9.84 & 19.48 & 1 & 7013 & 892702 & 3.90 & 16.62 & 10 & 1643 & 321699 & 0.00 & 11.16 \\
		\multicolumn{3}{l}{\textbf{Average}}&& \textbf{1888} & \textbf{385124} & \textbf{1.71} & \textbf{35.69}&& \textbf{983} & \textbf{178076} & \textbf{0.44} & \textbf{30.27}&& \textbf{239} & \textbf{61116} & \textbf{0.00} & \textbf{22.02}\\
		\multicolumn{3}{l}{\textbf{Allsolved}} & \textbf{147} &&&&& \textbf{164} &&&&& \textbf{180} &&&&\\
		\bottomrule
	\end{tabular}}}
    \label{tablesocp}
\end{table}
Next, we demonstrate the strength of the proposed \LEPI's for solving problem \eqref{problemSOCP}.
We generate \MPKPG instances using a procedure similar to that of \citet{Atamturk2013}.
Specifically, for each instance,
\begin{itemize}
	\item The numbers of variables $|N|$ and  probabilistic knapsack constraints $|M|$ are taken from $\{80, 120, 160\}$ and $\{20,30,40\}$, respectively.
	\item The sets $\{Q_k\}_{k\in K}$ are set to be disjoint with $\bigcup_{k\in K}Q_k=N$, where each $|Q_k|$ is randomly chosen from $[0.05|N|, 0.10|N|]$;
	\item The reliability $\rho$ is set to $0.95$; the parameters $c_i$, $a_{im}$, and $\sigma_{im}$ are set to be integers uniformly chosen from $[1, 1000]$, $[1, 100]$, and $[1, 2a_{im}]$, respectively; and
	$b_m=\beta\cdot(\sum_{k\in K} \max_{i\in Q_k}\{a_{im}\}
	+\Phi^{-1}(\rho)\sqrt{\sum_{k\in K} (\max_{i\in Q_k}\{\sigma_{im}\})^2})$, where $\beta$ is a parameter taken from $\beta\in\{0.3, 0.5\}$.
\end{itemize}  
For each combination of $|N|$, $|M|$, and $\beta$, we generate 10 random instances, yielding a testbed of $180$ instances.

Note that problem \eqref{problemSOCP} is a mixed integer second order conic programming problem, and hence can be directly solved by CPLEX.
Also note that based on the substructure $X^m$ in \eqref{substructureXm}, we can construct the proposed \LEPI's and the \EPI's, and use them to solve formulation \eqref{problemSOCP}.
To illustrate the advantages of the proposed \LEPI's, we compare the performance of the following three settings:

\begin{itemize}
	\item \CPX: using CPLEX to solve the \SOCP formulation \eqref{problemSOCP},
	\item \CPXEPI: \CPX with \EPI's,
	\item \CPXLEPI:  \CPX with \LEPI's.
\end{itemize}

\cref{tablesocp} summarizes the performance results.
From \cref{tablesocp}, we first observe that the average continuous relaxation gap at the root node under setting \CPX is $35.69\%$ while that under setting \CPXEPI is $30.27\%$, which shows that \EPI's can indeed strengthen the continuous relaxation of problem \eqref{problemSOCP}.
Thus, \CPXEPI achieves an overall better performance than \CPX.
Note that the average continuous relaxation gap at the root node under the setting \CPXLEPI is only $22.02\%$,
which shows that the \LEPI's provide significantly tighter continuous relaxation of problem \eqref{problemSOCP} than \EPI's.
This, again, confirms the strength of the proposed \LEPI's over \EPI's; see also \cref{tmpcoro}. 
Due to this advantage,  \CPXLEPI significantly outperforms  \CPXEPI. 
\CPXLEPI can solve all instances to optimality with the average CPU time being $239$ seconds and the average number of nodes being $61116$, while \CPXEPI can only solve $164$ instances to optimality with the average CPU time being  $983$ seconds and the average number of nodes being $178076$.
Notably, for the case with $\beta=0.5$, $n=160$, and $m=40$, \CPXLEPI can solve all $10$ instances to optimality while \CPXEPI can only solve $1$ instance to optimality within the time limit of $2$ hours.

\section{Concluding Remarks}\label{section8}
In this paper, we have studied the polyhedral structure of the submodular sets with \GUB constraints: $X_0=\{ (w,x) \in \R \times\{0,1\}^{n}\,:\, 
w \geq f(a^\top x),~ \sum_{i \in N_k} x_i \leq 1,~ \forall~ k\in [t]\}$ and its generalization $X=\{ (w,x) \in \R \times\{0,1\}^{n}\,:\, 
w \geq f(a^\top x)+b^\top x,~ \sum_{i \in N_k} x_i \leq 1,~ \forall~ k\in [t]\}$, which arise as important substructures in many applications.
We have developed a class of strong valid inequalities for the two sets, called  \LEPIs, using sequential lifting techniques.
Three key features, which make them particularly suitable to be embedded in a branch-and-cut framework to solve related \MINLP problems, are as follows.
First, they are facet-defining for the convex hulls of the two sets and are stronger than the well-known \EPIs of \citet{Edmonds2003}. 
Second, together with the bound and \GUB constraints, they are able to provide complete linear descriptions of the convex hulls of the two sets.
Third, they can be separated using an efficient $\CO (n\log n)$-time  algorithm.
By extensive computational experiments on \cGCMCLP{s} and \MPKPG{s}, we have demonstrated that compared with the classic \EPIs, the proposed \LEPIs are much more effective in strengthening the continuous relaxations and  improving the overall computational performance of the branch-and-cut algorithms.

A promising direction for future research is to study a more general submodular set with \GUB constraints, namely, $X^g=\{ (w,x) \in \R \times\{0,1\}^{n}\,:\, 
w \geq g(x),~ \sum_{i \in N_k} x_i \leq 1,~ \forall~ k\in [t]\}$, where $g$ is an arbitrary submodular function.
Studying the polyhedral structure of $\conv(X^g)$, however, poses more challenges compared to those of $\conv(X_0)$ and $\conv(X)$.
In particular, as shown\label{key} in Appendix \ref{appendixC}, the submodular minimization with {disjoint} GUB constraints, namely $\min_{x\in \{0,1\}^n}\left\{g(x): \sum_{i\in N_k}x_i\leq 1, ~\forall~ k\in [t]\right\}$, is NP-hard, and therefore one cannot  expect to  obtain a tractable linear description for $\conv(X^g)$.
Nevertheless, it would be still interesting to investigate the sequential lifting procedure for $X^g$ that takes the \GUB constraints into consideration towards  developing strong valid inequalities to provide a partial characterization of  $\conv(X^g)$. 

\newpage
\begin{appendices}
\section{A proof of \cref{lem:D1-D2}}\label{appendixA}
\begin{proof}
By \cref{prop:W-equals-U}(iv), it suffices to show $\{U^{\D^1}_j\}_{j\in [n]}\neq \{U^{\D^2}_j\}_{j\in [n]}$.
Let $\ell$ be the smallest index such that $\D^1_{\ell}\neq \D^2_{\ell}$.
Then $\D^1(j)=\D^2(j)$ for all $j < \ell$. 
Without loss of generality, we assume that $\D^1_{\ell}>\D^2_{\ell}$.
Let $k_1, k_2\in [t]$ be such that $\D^1_\ell \in N_{k_1}$ and $\D^2_\ell\in N_{k_2}$.
If $k_1 = k_2$, then by $\D^2_{\ell}\notin \D^2(\ell-1)=\D^1(\ell-1)$, $\D^2_{\ell}\neq \D^1_{\ell}$, and $\D^1\in \Delta$, $\D^2_{\ell}>\D^1_{\ell}$ must hold, a contradiction with $\D^1_{\ell}>\D^2_{\ell}$.
Therefore, $k_1\neq k_2$ must hold. 

For	any $j\leq \ell-1$, we have   
\begin{equation*}
	\max_{i\in N_{k_1}\cap \D^2(j)}i 
	= \max_{i\in N_{k_1}\cap \D^1(j)}i
	< 
	\max_{i\in N_{k_1}\cap \D^1(\ell)}i=\D^1_{\ell}
\end{equation*} 
and therefore, it follows from the definitions of $U^{\D^1}_{\ell}, U^{\D^2}_j$ in \eqref{def:U} that $U^{\D^1}_{\ell}\neq U^{\D^2}_j$. 
For any $j\geq \ell$, we have  
\begin{equation*}
	\max_{i\in N_{k_2}\cap \D^2(j)}i\geq 	\max_{i\in N_{k_2}\cap \D^2(\ell)}i= \D^2_{\ell}>\max_{i\in N_{k_2}\cap \D^2(\ell-1)}i=\max_{i\in N_{k_2}\cap \D^1(\ell-1)}i
	\stackrel{(a)}{=}\max_{i\in N_{k_2}\cap \D^1(\ell)}i
\end{equation*}
where (a) follows from $\D^1_{\ell}\notin N_{k_2}$ (as $\D^1_{\ell}\in N_{k_1}$  and $k_1\neq k_2$). 
Similarly, it follows from the definition of $U^{\D^1}_{\ell}, U^{\D^2}_j$ in \eqref{def:U} that $U^{\D^1}_{\ell}\neq U^{\D^2}_j$. 
Consequently, $U_\ell^{\D^1} \notin \{U^{\D^2}_j\}_{j\in [n]}$ and  hence $\{U^{\D^1}_j\}_{j\in [n]}\neq \{U^{\D^2}_j\}_{j\in [n]}$.
\end{proof}

\section{An $\mathcal{O}(n^3)$ algorithm for solving problem \eqref{opt}}\label{appendixB}
	Here, we present an $\mathcal{O}(n^3)$ algorithm for solving problem \eqref{opt}; see \citet{Hassin1989} or \citet{Atamturk2009} for similar algorithms for a class of submodular minimization problems or with a cardinality constraint, respectively. 
	
	We first consider the two simple cases $d<0$ and $d=0$ for which solving \eqref{opt} is easy.
	For $d<0$, as $(w,x)=(1,\boldsymbol{0})$ is a ray of $\conv\{(w,x)\in \R\times \{0,1\}^n\,:\, 
	w\geq f(a^\top x), ~\sum_{i\in N_k}x_i\leq 1, ~\forall~ k\in [t]\}$, problem \eqref{opt} must be unbounded.
	For $d=0$, it follows that
	$$\min_{x\in \{0,1\}^n}\left\{c^\top x\,:\, \sum_{i\in N_k}x_i\leq 1, ~\forall~ k\in [t]\right\}
	= \sum_{k = 1}^{t} \min\left\{\min_{i\in N_k} c_i, 0\right\},$$
    and thus problem \eqref{opt} can be solved in $\CO(n)$ time.
	
	Next we consider the case $d>0$.
	Since $w=f(a^\top x)$ must hold for each optimal solution $(w,x)$ of problem \eqref{opt}, we can project variable $w$ out from problem \eqref{opt} and rewrite \eqref{opt} as a set optimization problem.
	
	\begin{equation}\label{equiopt1}
		\min_{S\subseteq [n]}\{d f(a(S))+c(S)\,:\,|S\cap N_k|\leq 1, ~\forall~ k\in [t]\}.
	\end{equation}
	Letting $Z=\{(a(S), c(S))\,:\, ~S\subseteq [n], ~|S\cap N_k|\leq 1, ~\forall~ k\in [t]\}$, then problem \eqref{equiopt1} is equivalent to 
	$	\min_{(\alpha, \beta) \in Z} d f(\alpha) + \beta$.
	Consider a relaxation problem of \eqref{equiopt1}:
	\begin{equation}\label{equiopt2}
		\min_{(\alpha, \beta) \in \conv(Z)} d f(\alpha) + \beta.
	\end{equation}
	Since $f$ is concave and $d >0$, $d f(\alpha) + \beta$ is concave on $(\alpha, \beta)$.
	As a result, problem \eqref{equiopt2} must have an optimal solution that is an extreme point of $\conv(Z)$, and therefore, 
	problem \eqref{equiopt2} is equivalent to \eqref{equiopt1}.
	
	To solve problem \eqref{equiopt2}, it suffices to enumerate all extreme points of $\conv(Z)$, which can be done by finding an optimal solution to
	\begin{equation}\label{optX}
		\min_{(\alpha, \beta) \in Z} \lambda_1 \alpha + \lambda_2 \beta
	\end{equation}
	for each $(\lambda_1, \lambda_2)\in \R^2\backslash \{(0,0)\}$. 
	Observe that 
	$\min_{(\alpha, \beta) \in  Z} \lambda_1 \alpha + \lambda_2 \beta 
	 = \min_{S\subseteq [n]}\{\lambda_1a(S)+\lambda_2 c(S)\,:\, |S\cap N_k|\leq 1, ~\forall~ k\in [t]\}
	=\sum_{k=1}^t \min\left\{\min_{i\in N_k} \left\{ \lambda_1a_i+\lambda_2 c_i \right\},0\right\}$.
	Therefore, for a fixed $(\lambda_1, \lambda_2)\in \R^2\backslash \{(0,0)\}$, we can construct an optimal solution 
	\begin{equation}\label{candicate}
		\begin{aligned}
			& (\alpha^*,\beta^*)=(a(S^*), c(S^*)), ~\text{where}~S^* = \{ \ell_k \, : \, \lambda_1a_{\ell_k}+\lambda_2c_{\ell_k} < 0, ~k \in [t]  \},\\
			 & \qquad\qquad\quad
			 \ell_k= \min\left\{i_0\in N_k \,:\, \lambda_1 a_{i_0} +\lambda_2 c_{i_0}  = \min\limits_{i\in N_k} \left\{ \lambda_1a_i+\lambda_2c_i \right\}\right\}, ~\forall~k \in [t].
		\end{aligned}
	\end{equation}
	in $\CO(n)$ time.
	In the following, we will show that to solve problems \eqref{optX} with different $(\lambda_1, \lambda_2)\in \R^2\backslash \{(0,0)\}$, it suffices to consider at most $\mathcal{O}(n^2)$ candidate solutions. 
	This implies that
	there are at most $\mathcal{O}(n^2)$ candidate extreme points of $\conv(Z)$, which, together with the fact that computing a solution $(a(S^*), c(S^*))$ can be performed in $\mathcal{O}(n)$ time, implies that problem \eqref{equiopt1} can be solved in $\mathcal{O}(n^3)$ time.
	
	We consider the following four cases of problem \eqref{optX}. 
	\begin{itemize}
		\item [(i)] $\lambda_1=0$ and $\lambda_2 > 0$. Then, one optimal solution of problem \eqref{optX} is given by $(\alpha^*,\beta^*)=(a(S^*), c(S^*))$ where $S^* = \{ \ell_k \, : c_{\ell_k} < 0, ~k \in [t]  \}$ and $ \ell_k= \min\{i_0\in N_k \,:\,  c_{i_0}  = \min\limits_{i\in N_k} \left\{c_i \right\}\}$  for $k \in [t]$.
		\item [(ii)] $\lambda_1=0$ and $\lambda_2 < 0$. Similarly, 
		one optimal solution of problem \eqref{optX} is given by $(\alpha^*,\beta^*)=(a(S^*), c(S^*))$ where $S^* = \{ \ell_k \, : -c_{\ell_k} < 0, ~k \in [t]  \}$ and $ \ell_k= \min\{i_0\in N_k \,:\, - c_{i_0}  = \min\limits_{i\in N_k} \left\{ -c_i \right\}\}$  for $k \in [t]$.
		\item [(iii)] $\lambda_1>0$. Letting $\theta=\frac{\lambda_2}{\lambda_1} \in \R$, then problem \eqref{optX} is equivalent to $\min_{(\alpha, \beta) \in Z}  \alpha + \theta \beta$ and $\ell_k$ in \eqref{candicate} reduces to
		$\ell_k = \min\{i_0\in N_k \,:\,  a_{i_0} +\theta c_{i_0}  = \min\limits_{i\in N_k} \left\{ a_i+\theta c_i \right\}\}$.
		Let 
		$$
			\Theta = \bigcup_{k \in [t]} \left( \left\{ \frac{a_i-a_j}{c_j - c_i}\, : \, c_j \neq c_i,~\forall~i,j \in N_k  \right\}\cup \left\{ -\frac{a_i}{c_i} \, : \, c_i \neq 0,~\forall~i \in N_k \right\}\right),
		$$
		and $\theta_1, \theta_2, \ldots, \theta_\tau $ be such that $ \Theta = \{ \theta_1 ,\theta_2, \ldots \theta_\tau \}$ and $\theta_1 < \theta_2 < \cdots < \theta_\tau$. 
		Denote $\theta'_{0} = \theta'_{1}-1$, $\theta'_{i}:=\frac{\theta_i+ \theta_{i+1}}{2}$ for $i \in [\tau-1]$, and $\theta'_{\tau}:=\theta_{\tau}+1$.
		Observe that 
		\begin{itemize}
		\item [(1)] for $\theta  \in (\theta_i, \theta_{i+1})$ with $i\in [\tau-1]$, problem $\min_{(\alpha, \beta) \in Z}  \alpha + \theta \beta$ has an optimal solution that is identical to an optimal solution of  $\min_{(\alpha, \beta) \in Z}  \alpha + \theta'_{i} \beta$ (i.e.,  $(a(S^*), c(S^*))$ in \eqref{candicate} with $\frac{\lambda_2}{\lambda_1}=\theta'_i$); 
		\item [(2)] for $\theta \in (-\infty, \theta_{1})$, problem $\min_{(\alpha, \beta) \in Z}  \alpha + \theta \beta$ has an optimal solution that is identical to an optimal solution of  $\min_{(\alpha, \beta) \in Z}  \alpha + \theta'_0 \beta$ (i.e.,  $(a(S^*), c(S^*))$ in \eqref{candicate} with $\frac{\lambda_2}{\lambda_1}=\theta'_0$);
		\item [(3)] for $\theta \in (\theta_{\tau}, +\infty)$, problem $\min_{(\alpha, \beta) \in Z}  \alpha + \theta \beta$ has an optimal solution that is identical to an optimal solution of  $\min_{(\alpha, \beta) \in Z}  \alpha + \theta'_{\tau} \beta$ where (i.e.,  $(a(S^*), c(S^*))$ in \eqref{candicate} with $\frac{\lambda_2}{\lambda_1}=\theta'_{\tau}$);
		\item [(4)] for $\theta=\theta_i$ with $i\in [\tau]$, problem $\min_{(\alpha, \beta) \in Z}  \alpha + \theta \beta$ has an optimal solution that is identical to an optimal solution of  $\min_{(\alpha, \beta) \in Z}  \alpha + \theta' \beta$ where $\theta'=\theta'_{i-1}$ or $\theta'_i$ (i.e.,  $(a(S^*), c(S^*))$ in \eqref{candicate} with $\frac{\lambda_2}{\lambda_1}=\theta'_{i-1}$ or $\theta'_i$).
		\end{itemize}
		Note that there are at most $1+\sum_{k=1}^t \binom{|N_k|+1}{2}=1+\sum_{k=1}^t(|N_k|+|N_k|^2)/2=1+(n+ \sum_{k=1}^t|N_k|^2)/2\leq (n+n^2)/2+1$ points in $\{\theta'_i\}_{i=0}^{\tau}$.
		Therefore, to solve problem  $\min_{(\alpha, \beta) \in Z}  \alpha + \theta \beta$ with arbitrary $\theta \in \mathbb{R}$, or equivalently, to solve problems \eqref{optX} with arbitrary $(\lambda_1, \lambda_2) \in \mathbb{R}^2$ satisfying  $\lambda_1 > 0$, we only need to consider the at most $(n+n^2)/2+1$ candidate solutions $(a(S^*), c(S^*))$ in \eqref{candicate} with $\frac{\lambda_2}{\lambda_1}=\theta'_i$ for some $i \in [\tau]\cup 0$.
		\item [(iv)] $\lambda_1<0$. Similar to case (iii), we can  show that to solve problems \eqref{optX} with arbitrary $(\lambda_1, \lambda_2) \in \mathbb{R}^2$ satisfying  $\lambda_1 < 0$, we only need to consider the at most $(n+n^2)/2+1$ candidate solutions.
	\end{itemize}

\section{Hardness of the submodular minimization with {disjoint} GUB constraints}\label{appendixC}
For notational convenience in the proof, we restate  the submodular minimization with {disjoint} GUB constraints, namely, $\min_{x\in \{0,1\}^n}\left\{g(x): \sum_{i\in N_k}x_i\leq 1, ~\forall~ k\in [t]\right\}$ as a set optimization problem:
\begin{equation}\label{prob:minsubmodular}
	\min_{S\subseteq [n]}\left\{g(S): |S\cap N_k|\leq 1, ~\forall~ k\in [t]\right\}.
\end{equation}
Here we refer to $g(\chi^S)$ as a set function $g(S)$ by abusing notation.
\begin{Theorem}\label{thm:nphard}
	Problem \eqref{prob:minsubmodular}
	is NP-hard, where $g:2^{[n]}\rightarrow \R$ is a submodular function and $\{N_k\}_{k\in[t]}$ is a partition of $[n]$.
\end{Theorem}
To prove \cref{thm:nphard}, we first define the following set function:
\begin{equation}\label{hdef}
	h(S):=\sum_{i\in O}(1-\mathbf{1}_{D_i\subseteq S}).
\end{equation}
where $\{D_i\}_{i\in O}$ is an arbitrary collection of subsets of $[n]$, $S\subseteq [n]$, and $\mathbf{1}_{D_i \subseteq S}=1$ is the indicator function that equals 1 if $D_i \subseteq S$
and $0$ otherwise.

\begin{Lemma}\label{lem:submodularh}
	The set function $h$ is submodular.
\end{Lemma}
\begin{proof}
	Let $\sigma_i(S)= 1-\mathbf{1}_{D_i\subseteq S}$ for $i \in O$. 
	We first show that $\sigma_i(S)$ is submodular, that is,
	\begin{equation}\label{supermodularindicator}
		\mathbf{1}_{D_i\subseteq (S\cup j)}- \mathbf{1}_{D_i\subseteq S}
		\leq \mathbf{1}_{D_i \subseteq (V\cup j)}- \mathbf{1}_{D_i\subseteq V}.
	\end{equation}
	holds for all $S\subseteq V \subseteq [n]$ and $j\in [n]\backslash V $.
	Observe that $\mathbf{1}_{D_i\subseteq (S\cup j)} \geq \mathbf{1}_{D_i\subseteq S}$ and $\mathbf{1}_{D_i\subseteq (V\cup j)} \geq \mathbf{1}_{D_i\subseteq V}$. 
	Thus, \eqref{supermodularindicator} holds if  $\mathbf{1}_{D_i\subseteq (S\cup j)} = \mathbf{1}_{D_i\subseteq S}$.
	Otherwise, $\mathbf{1}_{D_i\subseteq (S\cup j)}=1$ and $\mathbf{1}_{D_i\subseteq S}=0$, which implies $D_i\subseteq S\cup j \subseteq V\cup j$ and $j \in D_i$.
	This, together with $j \notin V$, implies $\mathbf{1}_{D_i \subseteq (V\cup j)}- \mathbf{1}_{D_i\subseteq V}=1-0=1$.
	This shows the submodularity of $\sigma_i(S)$ for $i \in O$.
	Finally, the submodularity of $h(S)$ follows from the submodularity of $\{\sigma_i(S)\}_{i \in O}$ and  the fact that any non-negative linear combination of submodular functions is submodular. 
\end{proof}
We now come to prove \cref{thm:nphard}.
\begin{proof}{(Proof of \cref{thm:nphard})}
	 We show that the decision version of problem \eqref{prob:minsubmodular}
	 \begin{quote}
	 	Given a rational value $\tau\in \mathbb{Q}$, a partition $\{N_k\}_{k \in [t]}$ of $[n]$,
 		and a submodular function $g:2^{[n]}\rightarrow \R$, does there exist a subset $S \subseteq [n]$ such that $g(S) \leq \tau $ and $ |S\cap N_k|\leq 1$ for all $ k\in [t]$?
	 \end{quote}
	 is NP-complete by reduction from the NP-complete  \emph{independent set problem} (\ISP) \citep{Garey1979}.
	Given a graph $G=(V,E)$ and a positive integer $\ell \leq |V|$, the \ISP asks to decide whether there exists an independent set $I$ such that $|I|\geq \ell$ (an independent set $I$ is a subset of $V$ such that for any $i_1, i_2\in I$, it follows $(i_1, i_2) \notin E$).
	Without loss of generality, we assume that $V=[m]$.
	Given an instance of the \ISP, we construct an instance of decision version of problem \eqref{prob:minsubmodular} by setting $n=2m$, $t=m$, $N_k=\{k, k+m\}$ for $k\in [m]$, $O=[m]$, $D_i=\{j+m\,:\, (i,j)\in E\}\cup \{i\}$ for $i\in [m]$, 
 	$g(S):=h(S)=\sum_{i=1}^m(1-\mathbf{1}_{D_i\subseteq S})$, and $\tau = m-\ell$.
 	By  \cref{lem:submodularh}, $h$ is a submodular function.
	In the following, we shall show that the answer to the decision version of problem \eqref{prob:minsubmodular} is yes  if and only if the answer to the \ISP is yes.
	
	Suppose that the answer to the \ISP is yes, i.e.,
	there exists an independent set $I$ such that $|I|\geq \ell$.
	Let $S'=I\cup \{j+m\,:\, j\in [m]\backslash I\}$.
	It follows from $N_k = \{k, k +m\}$ that $|S'\cap N_k|\leq 1$ for all $k\in [m]$.
	To show that the answer to the decision version of problem \eqref{prob:minsubmodular} is yes, it suffices to prove $h(S')\leq m-\ell$.
	Letting $i\in I$, then it follows from the definition of the independent set that for all $j$ with $(i,j)\in E$, $j\in [m]\backslash I$ must hold.
	Thus, $D_i=\{j+m\,:\, (i,j)\in E\}\cup \{i\} \subseteq \{j+m\,:\, j\in [m]\backslash I\}\cup \{i\} \subseteq \{j+m\,:\, j\in [m]\backslash I \}\cup I=S'$.
	This means that $\mathbf{1}_{D_i\subseteq S'}=1$ for all $i\in I$, and thus
	\begin{equation*}
		h(S')=\sum_{i=1}^m(1-\mathbf{1}_{D_i\subseteq S'})=m-\sum_{i=1}^m \mathbf{1}_{D_i\subseteq S'}
		\leq m-\sum_{i\in I}\mathbf{1}_{D_i\subseteq S'}=m-|I|\leq m-\ell.
	\end{equation*} 
	
	Suppose that the answer to the decision version of problem \eqref{prob:minsubmodular} is yes, i.e., there exists a subset $S'\subseteq [2m]$ such that $|S'\cap N_k|\leq 1$ for $k\in [m]$ and  $h(S')\leq m-\ell$.
	Denote $I=\{i\in [m]\,:\, D_i\subseteq S'\}$.
	Then $|I|=\sum_{i=1}^m \mathbf{1}_{D_i\subseteq S'}=m-h(S')\geq m-(m-\ell)=\ell$.
	Suppose that $(i_1, i_2)\in E$ holds for some $i_1, i_2\in I$.
	Then, by the definitions of $D_{i_1}$, $D_{i_2}$ and $I$, we have $i_1\in D_{i_1}\subseteq S'$ and $i_1 +m \in \{j+m\,:\, (i_2,j)\in E\}\subseteq D_{i_2}\subseteq S'$.
	Thus, $i_1,i_1+m\in S'$ and $|S'\cap N_{i_1}| = |\{i_1, i_1+m\}|=2$, a contradiction with $|S'\cap N_{i_1}|\leq 1$.
	Therefore, for any $i_1, i_2\in I$, it follows $(i_1, i_2) \notin E$, and thus, $I$ is an independent set with $|I|\geq \ell$, yielding a yes answer to the \ISP.	
\end{proof}
\end{appendices}

\bibliography{shorttitles,gub_submodular}{}
\bibliographystyle{apalike}

\end{document}